\documentclass[10pt,a4paper]{article}
\usepackage{graphicx}
\usepackage{subcaption}
\usepackage{xcolor}
\usepackage{comment}
\usepackage[utf8]{inputenc}
\usepackage{comment}
\usepackage{authblk}

\usepackage{hyperref}
\hypersetup{
    colorlinks=true,
    linkcolor=blue,
    citecolor = magenta,
    filecolor=magenta,      
    urlcolor=cyan,
}
\usepackage{siunitx,booktabs}
\usepackage{arydshln}

\usepackage{setspace,amsmath,amsthm,amsfonts,amssymb,mathtools,bm,algorithm,algpseudocode,
longtable,multirow,geometry,mathrsfs,tablefootnote,threeparttable,dsfont}
\usepackage[font=sf, labelfont={sf,bf}, margin=1cm]{caption}
\newtheorem{Proposition}{Proposition}

\numberwithin{Theorem}{section}
\numberwithin{Definition}{section}
\numberwithin{Lemma}{section}
\numberwithin{Algorithm}{section}
\numberwithin{equation}{section}

\newtheorem{theorem}{Theorem}[section]

\newtheorem{lemma}[theorem]{Lemma} 
\newtheorem{assumption}{Assumption}

\newtheorem{remark}{Remark}

\title{An efficient active-set method with applications to sparse approximations and risk minimization}
\author[1]{Spyridon Pougkakiotis}
\author[2]{Jacek Gondzio}
\author[3]{Dionysis Kalogerias}
\affil[1]{School of Science and Engineering, University of Dundee, UK}
\affil[2]{School of Mathematics, University of Edinburgh, UK}
\affil[3]{Department of Electrical Engineering, Yale University, USA}

\makeatletter
\def\@cline#1-#2\@nil{%
  \omit
  \@multicnt#1%
  \advance\@multispan\m@ne
  \ifnum\@multicnt=\@ne\@firstofone{&\omit}\fi
  \@multicnt#2%
  \advance\@multicnt-#1%
  \advance\@multispan\@ne
  \leaders\hrule\@height\arrayrulewidth\hfill
  \cr
  \noalign{\nobreak\vskip-\arrayrulewidth}}
\makeatother
\begin{document}

\maketitle

\begin{abstract}
\par In this paper we present an efficient active-set method for the solution of convex quadratic programming problems with general piecewise-linear terms in the objective, with applications to sparse approximations and risk-minimization. The algorithm is derived by combining a proximal method of multipliers (PMM) with a standard semismooth Newton method (SSN), and is shown to be globally convergent under minimal assumptions. Further local linear (and potentially superlinear) convergence is shown under standard additional conditions. The major computational bottleneck of the proposed approach arises from the solution of the associated SSN linear systems. These are solved using a Krylov-subspace method, accelerated by certain novel general-purpose preconditioners which are shown to be optimal with respect to the proximal penalty parameters. The preconditioners are easy to store and invert, since they  exploit the structure of the nonsmooth terms appearing in the problem's objective to significantly reduce their memory requirements. We showcase the efficiency, robustness, and scalability of the proposed solver on a variety of problems arising in risk-averse portfolio selection, $L^1$-regularized partial differential equation constrained optimization, quantile regression, and binary classification via linear support vector machines. We provide computational evidence, on real-world datasets, to demonstrate the ability of the solver to efficiently and competitively handle a diverse set of medium- and large-scale optimization instances.
\end{abstract}

\section{Introduction}
\subsection{Problem formulation}
\par In this paper we consider convex optimization problems of the following form:
\begin{equation}\label{primal problem} \tag{P}
\begin{split}
\underset{x \in \mathbb{R}^n}{\text{min}}&\quad \left\{c^\top x + \frac{1}{2}x^\top Q x + \sum_{i = 1}^l \left( \left(C x + d\right)_{i}\right)_+ + \delta_{\mathcal{K}}(x)\right\}, \\ 
\textnormal{s.t.} &\quad \ A x = b,
\end{split}
\end{equation}
\noindent where $c \in \mathbb{R}^n$, $Q \in \mathbb{R}^{n \times n}$ is a positive semi-definite matrix, $C \in \mathbb{R}^{l\times n}$, $d \in \mathbb{R}^l$, $A \in \mathbb{R}^{m\times n}$ is a linear constraint matrix with $b\in \mathbb{R}^m$ a given right-hand side, and $\mathcal{K}$ is a closed convex set $\mathcal{K} \triangleq \{x \in \mathbb{R}^n \vert\ x \in \left[a_l,a_u\right]\}$, with $a_l,\ a_u \in \mathbb{R}^n$, such that $(a_l)_i \in \mathbb{R}\cup\{-\infty\},\ (a_u)_i \in \mathbb{R}\cup \{+\infty\}$, for all $i = 1,\ldots,n$. Additionally, $(\cdot)_+ \equiv \max\{\cdot,0\}$, while $\delta_{\mathcal{K}}(\cdot)$ is an indicator function for the set $\mathcal{K}$, that is, $\delta_{\mathcal{K}}(x) = 0$ if $x \in \mathcal{K}$, and $\delta_{\mathcal{K}}(x) = \infty$ otherwise.
\par For simplicity, we introduce some additional notation and equivalently write \eqref{primal problem} as
\begin{equation*} 
\begin{split}
\underset{x \in \mathbb{R}^n}{\text{min}} & \left\{f(x) +  h(Cx+d)  + \delta_{\mathcal{K}}(x)\right\},\\
\textnormal{s.t.} &\  Ax = b.
\end{split}
\end{equation*}
\noindent where we define $f(x) \triangleq c^\top x + \frac{1}{2}x^\top Q x$, and $h(w) \triangleq \sum_{i=1}^{l} (w_i)_+$.
\begin{remark} \label{remark: l1 norm reformulation}
\par Let us notice that the model in \eqref{primal problem} can readily accommodate terms of the form of $\|Cx + d\|_1$ where $C \in \mathbb{R}^{l\times n}$, and $d \in \mathbb{R}^l$. Indeed, letting $c = c - C^\top \mathds{1}_l$ and adding the constant term $-\mathds{1}_l^\top d$ in the objective of \eqref{primal problem}, we notice that
\[\|Cx + d\|_1 \equiv -\mathds{1}_l^\top (Cx + d) + \sum_{i=1}^l \left(2(Cx + d)_i\right)_+, \]
where $\mathds{1}_l \triangleq (1,\ldots,1)^\top \in \mathbb{R}^l$. Similarly, any piecewise-linear term of the form 
\[\sum_{i=1}^l\max\left\{\left(C_1 x+ d_1\right)_i, \left(C_2 x+ d_2\right)_i\right\},\]
\noindent where $C_1,\ C_2 \in \mathbb{R}^{l\times n}$ and $d_1,\ d_2 \in \mathbb{R}^l$, can also be readily modeled. Specifically, setting $c = c + C_2^\top\mathds{1}_l$ and adding the constant term $d_2^\top \mathds{1}_l$ in the objective yields
\[ \mathds{1}_l^\top \left(C_2 x + d_2\right) + \sum_{i=1}^l \left(\left(C_1 x + d_1 - C_2x - d_2\right)_i \right)_+ \equiv \sum_{i = 1}^l \max\left\{\left(C_1 x + d_1\right)_i,\left(C_2 x + d_2\right)_i \right\}.\]
\noindent Finally, it is important to note that model \eqref{primal problem} allows for multiple piecewise-linear terms of the form $\max\{Cx + d,0_l\}$, $\|Cx + d\|_1$ or $\max\left\{C_1 x+ d_1, C_2 x+ d_2\right\}$, since we can always adjust $l$ to account for more than one terms. Hence, one can observe that \eqref{primal problem} is very general and can be used to model a plethora of very important problems that arise in practice. 
\end{remark}
\par Let $F \triangleq [A^\top\ -C^\top]^\top$ and $\hat{b} \triangleq [b^\top\ d^\top]^\top$. Using Fenchel duality, we can verify (see Appendix \ref{Appendix: derivation of dual}) that the dual of \eqref{primal problem} is
\begin{equation} \label{dual problem} \tag{D}
\begin{split}
\underset{x \in \mathbb{R}^n, y \in \mathbb{R}^{m+l}, z \in \mathbb{R}^n}{\text{max}} &\  \hat{b}^\top y - \frac{1}{2} x^\top Q x - \delta^*_{\mathcal{K}}(z) - h^*(y_{m+1:m+l}),\\
\textnormal{s.t.}\qquad \ \ &\ c + Qx -F^\top y + z = 0_n.
\end{split}
\end{equation}
\noindent Throughout the paper we make use of the following blanket assumption.
\begin{assumption} \label{assumption: solution of the QP}
Problems \eqref{primal problem} and \eqref{dual problem} are both feasible and $A$ has full row rank.
\end{assumption}
\noindent If $C = 0_{0,n}$, using \cite[Proposition 2.3.4]{BertsekasNedicOzdaglar} we note that Assumption \ref{assumption: solution of the QP} implies that there exists a primal-dual triple $(x^*,y^*,z^*)$ solving \eqref{primal problem}--\eqref{dual problem}. If the primal-dual pair \eqref{primal problem}--\eqref{dual problem} is feasible, it must remain feasible for any $C \in \mathbb{R}^{l\times n}$, since in this case \eqref{primal problem} can be written as a convex quadratic problem by appending appropriate (necessarily feasible) linear inequality constraints. Thus, Assumption \ref{assumption: solution of the QP} suffices to guarantee that the solution set of \eqref{primal problem}--\eqref{dual problem} is non-empty.
\par  In light of the discussion in Remark \ref{remark: l1 norm reformulation}, we observe that \textbf{problem \eqref{primal problem} models a plethora of very important problems arising in several application domains} spanning, among others, operational research, machine learning, data science, and engineering. More specifically, \eqref{primal problem} can model linear and convex quadratic programming instances, regularized (fused) lasso instances (often arising in signal or image processing and machine learning, e.g., see \cite{SIREV:Chenetal,SIREV:DeSimone_etal,bookVapnik,JRSS:Zou}), or portfolio allocation problems \cite{JBanFin:AlexColYi} arising in finance. Furthermore, various optimal control problems can be tackled in the form of \eqref{primal problem}, such as those arising from $L^1$-regularized partial differential equation (PDE) constrained optimization, assuming that a \emph{discretize-then-optimize} strategy is adopted (e.g., see \cite{ESAIM:GerdDaniel}). Additionally, various risk-minimization problems with linear random cost functions can be modeled by \eqref{primal problem} (e.g., see \cite{MathFin:LiNg,JRisk:RockUry,Prod:Silva_etal}). Nonetheless, even risk-minimization problems with nonlinear random cost functions, which are typically solved via Gauss-Newton schemes (e.g., see \cite{MathProg:Burke}), often require the solution of sub-problems of the form of \eqref{primal problem}. Finally, continuous relaxations of integer programming problems with applications in operational research (e.g., \cite{MathProg:Laz}) often take the form of \eqref{primal problem}. Given the multitude of problems requiring easy access to (usually accurate) solutions of \eqref{primal problem}, the derivation of efficient, robust, and scalable solution methods is of paramount importance and has naturally attracted a lot of attention in the literature. 
\subsection{Existing solution methods}
\par There is a large variety of first-order methods capable of finding an approximate solution to \eqref{primal problem}. For example, one could employ proximal (sub-)gradient (e.g., see \cite{SIAM:Beck}) or splitting schemes (e.g., see \cite{SciComp:DengYin}). While such solution methods are very general, easy to implement, and require very little memory, they are usually able to find only highly approximate solutions, not exceeding 2- or 3-digits of accuracy. If a more accurate solution is needed, then one has to resort to an approach that utilizes second-order information. 
\par There are three major classes of second-order methods for problems of the form of \eqref{primal problem}. Those include globalized (smooth, semismooth, quasi or proximal) Newton methods (e.g., see \cite{MathOR:Han,CAM:MartQi,COAP:Stella_etal}), variants of the proximal point method (e.g., see \cite{COAP:Marchi,IEEE_CDC:Dhingra_etal,MathProgComp:Hermans_etal,SIAMOPT:Leeetal,SIAMOpt:Lietal}), or interior point methods (IPMs) applied to a reformulation of \eqref{primal problem} (e.g., see \cite{SIREV:DeSimone_etal,IPMs:FountoulakisEtAl2013,arXiv:GondPougkPears,NLAA:PearsonPorcStoll}). \par Most globalized Newton-like approaches or proximal point variants studied in the literature are developed for composite programming problems in which either $C = 0_{0,n}$ (e.g., see \cite{NLAA:ChenQi,COAP:Marchi,COAP:GillRobi,MathProg:Itoetal,SIAMOpt:Lietal}) or $\mathcal{K} = \mathbb{R}^n$ and the general $\max\{\cdot,0\}$ terms are substituted by separable $\ell_1$ terms (e.g., see \cite{IEEE_CDC:Dhingra_etal,InverseProbs:HansRaasch,SIAMOpt:Lietal2}). Nonetheless, there have been developed certain globalized Newton-like schemes, specialized to the case of $L^1$-regularized PDE-constrained optimization (see \cite{OptEng:MannelRund,CAA:Porcellietal}), in which both $\ell_1$ term as well as box constraints are explicitly handled. We should notice, however, that globalized Newton-like schemes applied to \eqref{primal problem} require additional assumptions (on top of Assumption \ref{assumption: solution of the QP}) to guarantee convergence or stability of the related Newton linear systems arising as sub-problems. Typically, superlinear convergence of Newton-like schemes is observed ``close to a solution". Under additional assumptions, global convergence can be achieved via appropriate line-search or trust-region strategies (e.g., see the developments in \cite{SIAMOpt:Christofetal,MathOR:Han,MathProg:Itoetal,SIAMOpt:Themelis_etal} and the references therein). 
\par The potential convergence and stability issues arising within Newton-like schemes can be alleviated by combining Newton-like methods with proximal point variants. Solvers based on the proximal point method can achieve superlinear convergence, assuming their penalty parameters increase indefinitely at a suitable rate (e.g., see \cite{MathOpRes:Rock,SIAMJCO:Rock}). 
For problems \eqref{primal problem}--\eqref{dual problem}, the sub-problems arising within proximal methods are (possibly nonsmooth) convex optimization instances, and are typically solved by means of semismooth Newton strategies. The associated SSN linear systems are better conditioned than their (possibly regularized) interior-point counterparts (e.g., see \cite{NLAA:BergGondMartPearPoug,SIREV:DeSimone_etal,IPMs:FountoulakisEtAl2013,arXiv:GondPougkPears,IPMs:WaltzMoralNocedOrban}); however, convergence is expected to be slower, as the method does not enjoy the polynomial worst-case complexity of interior-point methods. Nevertheless, these better conditioned linear systems can, in certain cases, allow one to achieve better computational and/or memory efficiency, especially if the nonsmooth terms are appropriately handled. 
\par Various such solvers have been developed and analyzed in the literature. For example, the authors in \cite{SIAMOpt:ZhaoSunToh} developed a dual augmented Lagrangian scheme combined with an SSN method for the solution of semidefinite programming problems and obtained very promising results. This scheme was then utilized for the solution of linear programming problems in \cite{SIAMOpt:Lietal}, and for lasso-regularized problems in \cite{SIAMOpt:Lietal2}. A similar dual augmented Lagrangian-based solver (ALM) was recently developed for square-root Lasso problems in \cite{WangTangSciComp}. Nonetheless, those aforementioned dual ALM solvers are brittle, in that slight alterations in the structure of the problem can significantly hinder their employability and efficiency. For example, a computationally attractive extension of \cite{SIAMOpt:Lietal} to convex quadratic (or, generally, non-separable) programming is by no means obvious and requires a different algorithmic strategy than the dualization approach followed by the authors. Furthermore, there is a \textbf{wide gap in the current literature of proximal-SSN methods concerning the efficient solution of the associated SSN linear systems}. While the convergence results appearing in the literature enable the use of inexact linear system solves (e.g., enabling the use of iterative linear algebra, such as Krylov subspace solvers), no general preconditioning strategies have been proposed or analyzed and the question of exploiting efficient linear algebra in this context is still open.
\subsection{Contributions}
\par In this paper, we develop an active-set method for \eqref{primal problem}--\eqref{dual problem} by employing an appropriate proximal method of multipliers (PMM) using a standard backtraking-linesearch semismooth Newton (SSN) strategy for solving the associated PMM sub-problems. The SSN linear systems are approximately solved by means of Krylov subspace methods, using certain novel general-purpose preconditioners. Unlike most proximal point methods given in the literature (e.g., see the primal approaches in \cite{SIAMOPT:Leeetal,IEEE_DC:Patrinos_etal}, the dual approaches in \cite{SIAMOpt:Lietal2,SIAMOpt:Lietal} or the primal-dual approaches in \cite{COAP:Marchi,COAP:GillRobi,MathOpRes:Rock}), the proposed method introduces  proximal terms for each primal and dual variable of the problem, and this results in Newton linear systems which are easy to precondition and solve. We dualize each of the two nonsmooth terms of the objective in \eqref{primal problem}, which contributes to the simplification of the resulting SSN linear systems, and paves the way for readily analyzing the convergence of the proposed algorithm. Our contributions are summarized as follows:
\begin{itemize}
    \item[$\bullet$] We are able to directly apply the theory developed in \cite{SIAMOpt:Lietal} to show that the proposed outer proximal method of multipliers is globally convergent under mere feasibility assumptions, and locally superlinearly convergent under standard additional assumptions. Importantly, we should mention that the theory in \cite{SIAMOpt:Lietal} was developed to analyze the convergence of a proximal semismooth Newton scheme designed for linear programming. Nonetheless, in this work we showcase that the proposed active-set scheme can be analyzed using this theory, and without requiring any additional assumptions other than those employed in \cite{SIAMOpt:Lietal} in the context of linear programming, despite our model \eqref{primal problem} being significantly more general. We also note the inherent difference of our approach compared to the dual augmented Lagrangian schemes developed in \cite{SIAMOpt:Lietal2,SIAMOpt:Lietal,WangTangSciComp,SIAMOpt:ZhaoSunToh}, which enables the efficient solution of problems with non-separable (smooth or nonsmooth) terms in the objective function of \eqref{primal problem}.
    \par Additionally, by introducing an appropriate augmented Lagrangian penalty function (see Section \ref{sec: PAL penalties}), we obtain continuously differentiable PMM sub-problems. As a result, we are able to apply a standard globally convergent (at a local superlinear rate) backtracking-linesearch semismooth Newton scheme (e.g., see \cite{SIAMOpt:ZhaoSunToh}) for their solution. Thus, our active-set scheme is readily theoretically supported, and covers a very wide range of problems not currently addressed by similar methods in the literature, unless \eqref{primal problem} is reformulated without exploiting its inherent structure.
    \item[$\bullet$] The price to pay for the strong theoretical properties of the proposed approach is that the difficulty arising from the nonsmooth terms is hidden in the solution of the associated SSN linear systems. Nonetheless, we identify a principled way of exploiting the resulting active-set structure, that occurs by utilizing the Clarke subdifferential within SSN, to create novel, robust, and highly efficient preconditioners for the solution of these linear systems by means of Krylov subspace solvers. Specifically, we derive and analyze highly effective preconditioning strategies which we show to be optimal with respect to the penalty parameters of the PMM; that is, the conditioning of the preconditioned linear systems does not deteriorate as we increase the penalty parameters of the PMM (which, in turn, accelerates the convergence of the PMM). Additionally, the preconditioners are versatile and general (i.e., not problem-dependent), and are very manageable in terms of memory requirements, giving a strong computational edge of the proposed approach compared to alternative second-order methods. The developed linear algebra kernel of our method showcases the advantages of utilizing active-set schemes compared to alternative (state-of-the-art) interior point methods for similar problems; such a study was missing from the field of augmented Lagrangian methods, and we argue that it should be capitalized more heavily in similar augmented Lagrangian-type approaches appearing in the literature. To the best of our knowledge, aside from the work in \cite{CAA:Porcellietal}, which is specialized to the case of $L^1$-regularized PDE constrained optimization, most proximal Newton-like schemes utilizing Krylov subspace methods do so without employing any (non-trivial) preconditioner (e.g., see \cite{NLAA:ChenQi,SIAMOpt:Lietal2,IEEE_DC:Patrinos_etal,WangTangSciComp}). 
    \item[$\bullet$] We provide extensive computational evidence to support that the proposed approach exhibits a significantly better practical convergence than that predicted by the theory, over a very wide range of applications. Specifically, we showcase the robustness and the efficiency of our approach over a plethora of instances arising from risk-averse portfolio selection problems, $L^1$-regularized partial differential equation constrained optimization, quantile regression as well as linear classification via support vector machines. In each of these cases, we compare our approach to a robust regularized interior point method (see \cite{arXiv:GondPougkPears,COAP:PougkGond}) which is available on GitHub, and showcase its efficiency not only in terms of CPU time and memory requirements, but also in terms of total iterations required to obtain a sufficiently accurate solution. Indeed, we demonstrate that while our approach requires more iterations compared to the alternative polynomially convergent interior point method, the discrepancy in terms of iteration count is not particularly significant. This is reflected in the computational efficiency of our approach, in terms of CPU time. At the same time, we demonstrate that both second-order solvers are significantly more robust and efficient compared to a state-of-the-art alternating direction method of multipliers \cite{osqp} (ADMM), since we are interested in finding sufficiently accurate solutions in our numerical experiments, which prevent first-order schemes from being competitive.
       \item[$\bullet$]The method is extremely versatile and can solve a plethora of very important applications (as we showcase in Section \ref{sec: applications}), by explicitly exploiting the structure of the nonsmooth terms appearing in the objective of \eqref{primal problem}. Additionally, it deals with general box constraints, without requiring the introduction of auxiliary variables to deal with upper and lower bounds separately, something that is required when utilizing conic-based solvers. As a result, the associated linear systems solved within SSN have significantly smaller dimensions, compared to linear systems arising within interior point methods suitable for the solution of problems like \eqref{primal problem} (e.g., see \cite{SIREV:DeSimone_etal,arXiv:GondPougkPears,NLAA:PearsonPorcStoll}), potentially making the proposed approach a more attractive alternative for very large-scale instances.
\end{itemize}
\par Before closing this section, we should mention that the algorithm developed in this paper arose in an attempt to consolidate two of our previously developed methodologies laid out in two technical reports (\cite{arxiv:Pougk_etal1,arxiv:Pougk_etal2}); the former report studies \eqref{primal problem} by assuming that the nonsmooth terms are separable, while the latter studies \eqref{primal problem} in its full generality. However, these previous approaches dealt with certain non-smooth terms of \eqref{primal problem} explicitly. As a result, the associated PMM sub-problems of the developed algorithms were nonsmooth, and that resulted in limited theoretical support of the inner SSN scheme utilized to tackle these sub-problems. This issue was overcome in this consolidated work, by transferring the difficulty to the SSN linear systems, and then utilizing our versatile iterative linear algebra kernel (and associated preconditioning strategy) to ensure that this new and improved version of the active-set schemes developed in \cite{arxiv:Pougk_etal1,arxiv:Pougk_etal2} is equally scalable but also more robust and better theoretically supported.
\par To summarize, in Section \ref{sec: PAL penalties} we derive a proximal method of multipliers and briefly discuss its convergence properties. Then, in Section \ref{sec: SSN method} we present a well-studied globally convergent backtracking-linesearch semismooth Newton scheme used to approximately solve the PMM sub-problems, noticing its local superlinear convergence properties under no additional assumptions. Next, in Section \ref{sec: solution of linear systems}, we propose \emph{novel} general-purpose preconditioners for the associated SSN linear systems and analyze their effectiveness. In Section \ref{sec: applications}, our approach is extensively tested on a wide variety of optimization instances. Finally, we derive some conclusions in Section \ref{sec: Conclusions}.

\paragraph{Notation.} Given a vector $x$ in $\mathbb{R}^n$, $\|x\|$ denotes the Euclidean norm. Letting $R \succ 0_{n,n}$ be a symmetric positive definite matrix, we denote $\|x\|_R^2 = x^\top R x$. Given a closed set $\mathcal{K} \subset \mathbb{R}^n$, we write $\Pi_{\mathcal{K}}(x) \triangleq \arg\min\{\|x-z\|\ \vert\ z \in \mathcal{K}\}$, while for any $R \succ 0_{n,n}$ we write $\textnormal{dist}_R(z,\mathcal{K}) \triangleq \inf_{z'\in \mathcal{K}}\|z-z'\|_R$ (with $\textnormal{dist}_I(z,\mathcal{K}) \equiv \textnormal{dist}(z,\mathcal{K})$). Given a proper and closed convex function $p \colon \mathbb{R}^n \rightarrow \mathbb{R}\cup\{+\infty\}$, we define $\textbf{prox}_p (u) \triangleq \arg\min_x \left\{ p(x) + \frac{1}{2}\|u-x\|^2\right\}$. Given an arbitrary rectangular matrix $A$, $\sigma_{\max}(A)$ denotes its maximum singular value. For an arbitrary square matrix $B$, $\lambda(B)$ is the set of eigenvalues of $B$ while $\lambda_{\max}(B)$ (resp., $\lambda_{\min}(B)$) denotes its maximum (resp., minimum) eigenvalue. Given an index set  $\mathcal{D}$, $|\mathcal{D}|$ denotes its cardinality. Given a rectangular matrix $A \in \mathbb{R}^{m \times n}$ and an index set $\mathcal{B} \subseteq \{1,\ldots,n\}$, we denote the columns of $A$, the indices of which belong to $\mathcal{B}$, as $A_{\mathcal{B}}$. Given a square matrix $Q \in \mathbb{R}^{n\times n}$, we denote the subset of columns and rows of $Q$, the indices of which belong to $\mathcal{B}$, as $\left(Q\right)_{\mathcal{B},\mathcal{B}}$. We denote by $\textnormal{Diag}(Q)$ the diagonal matrix with diagonal elements equal to those of $Q$. Given a matrix $A \in \mathbb{R}^{m\times n}$, and two index sets $\mathcal{B}_1 \subseteq \{1,\ldots,m\}$ and $\mathcal{B}_2 \subseteq \{1,\ldots,n\}$, we denote the subset of rows and columns of $A$, the indices of which belong to $\mathcal{B}_1,\ \mathcal{B}_2$ respectively, as $\left(A\right)_{\mathcal{B}_1,\mathcal{B}_2}$. Finally, given an arbitrary vector $d$ with $n$ components as well as some indices $1\leq i_1 \leq i_2 \leq n$, we denote by $d_{i_1:i_2}$ the vector $(d_{i_1},d_{i_1+1},\ldots,d_{i_2})$. To avoid confusion, indices $i$ are always used to denote entries of vectors or matrices, while $k$ and $j$ are reserved to denote iteration counters (outer (PMM) and inner (SSN), respectively).
\section{A primal-dual proximal method of multipliers} \label{sec: PAL penalties}
\par In what follows, we derive the proximal augmented Lagrangian penalty function corresponding to the primal problem \eqref{primal problem}. Using the latter, we derive a primal-dual PMM for solving the pair \eqref{primal problem}--\eqref{dual problem}. The convergence of this PMM scheme is subsequently analyzed, assuming that we are able to find sufficiently accurate solutions to its associated sub-problems. In the next section, we briefly present a standard semismooth Newton scheme for the solution of these sub-problems, with global and local superlinear convergence guarantees. 
\subsection{Derivation of the PMM} \label{subsec: derivation of the PMM}
\par We begin by deriving the Lagrangian associated to \eqref{primal problem}. First, we define the function \[\varphi(x) \triangleq f(x)  + h(Cx+d) + \delta_{\mathcal{K}}(x) + \delta_{\{0_m\}}\left(b-Ax\right).\]
\noindent Following the dualization strategy proposed in \cite[Chapter 11]{Springer:RockWets}, we let
\[\hat{\varphi}(x,u',w',v') \triangleq f(x) + h(Cx+d+w') + \delta_{\mathcal{K}}(x+v') + \delta_{\{0_m\}}\left(b-Ax + u'\right),\]
\noindent for which it holds that $\varphi(x) =\hat{\varphi}(x,0_m,0_l,0_n)$. Then, by letting $y = [y_1^\top\ y_2^\top]^\top$, the Lagrangian associated to \eqref{primal problem} reads:
\begin{equation*}
\begin{split}
\ell (x,y,z) &\triangleq \ \inf_{u',w',v'} \left\{\hat{\varphi}(x,u',w',v') - y_1^\top  u' -y_2^\top w' - z^\top  v'\right\} \\ &=\ f(x)  - \sup_{u'} \left\{y_1^\top  u' -\delta_{\{0_m\}}\left(b-Ax + u'\right) \right\}  \\&\qquad -\sup_{w'}\left\{y_2^\top w' -h\left(Cx+d+w'\right)\right\} - \sup_{v'} \left\{z^\top v' - \delta_{\mathcal{K}}(x+v') \right\}\\
&=\ f(x)  -y_1^\top(Ax-b) +y_2^\top(Cx+d) -h^*(y_2) +   z^\top  x - \delta^*_{\mathcal{K}}(z),
\end{split}
\end{equation*}
\noindent where we used the definition of the Fenchel conjugate. Before we proceed with the derivation of the augmented Lagrangian associated to $\ell(\cdot)$, let us observe that the Fenchel conjugate of $h(\cdot)$ is $h^*(u) = \delta_{\mathcal{C}}(u)$, where $\mathcal{C} \triangleq \left\{u \in \mathbb{R}^l \big\vert v_i \in [0,1],\ i \in \{1,\ldots,l\}\right\}.$ Given a penalty parameter $\beta > 0$, the augmented Lagrangian corresponding to \eqref{primal problem} reads:
\begin{equation} \label{final augmented lagrangian of the primal}
\begin{split}
\mathcal{L}_{\beta}(x; y, z) &\triangleq\  f(x)  - y_1^\top (Ax-b) + \frac{\beta}{2}\|Ax-b\|^2 \\
&\qquad +(Cx + d)^\top \Pi_{\mathcal{C}}\left(y_2 + \beta(Cx+d)\right) - \frac{1}{2\beta}\left\|y_2-\Pi_{\mathcal{C}}\left(y_2 + \beta(Cx+d)\right)\right\|^2\\
&\qquad -\frac{1}{2\beta}\|z\|^2 + \frac{1}{2\beta}\left\|z+\beta x -\beta \Pi_{\mathcal{K}}\left(\beta^{-1}z+ x\right)\right\|^2.
\end{split}
\end{equation}
\noindent A detailed derivation of the augmented Lagrangian penalty function given in \eqref{final augmented lagrangian of the primal} can be found in Appendix \ref{Appendix: derivation of the augmented Lagrangian}.
\par Assume that at iteration $k\geq 0 $ we have the estimates $(x_k,y_k,z_k)$ as well as the penalty parameters $\beta_k,\ \rho_k$, such that $\rho_k \triangleq \frac{\beta_k}{\tau_k}$, where $\{\tau_k\}_{k = 0}^{\infty}$ is a non-increasing positive sequence, i.e. $\tau_k > 0$ for all $k \geq 0$. We define the continuously differentiable function 
\begin{equation} \label{eqn: proximal Lagrangian penalty}
\phi(x) \equiv\  \phi_{\rho_k,\beta_k}(x;x_k,y_k,z_k) \triangleq  \mathcal{L}_{\beta_k}(x;y_k,z_k) + \frac{1}{2\rho_k}\|x-x_k\|^2.
\end{equation}
\noindent Using the previous notation, we need to find $x^*$ such that
\[\left(\nabla \phi(x^*)\right)^\top (x - x^*) \geq 0,\qquad \forall\ x \in \mathbb{R}^n,\]
\noindent where it can be shown (see Appendix \ref{Appendix: differentiability of augmented Lagrangian}) that
\begin{equation*}
\begin{split}
\nabla \phi(x) &= \nabla f(x) - A^\top y_{1_k} + \beta_kA^\top(Ax-b) + C^\top \Pi_{\mathcal{C}}\left(y_{2_k} + \beta_k\left( Cx+d\right) \right)\\
&\qquad  + (z_k + \beta_k x) - \beta_k \Pi_{\mathcal{K}}(\beta_k^{-1}z_k + x) + \rho_k^{-1}(x-x_k).
\end{split}
\end{equation*}
\noindent We now describe the primal-dual PMM in Algorithm \ref{primal-dual PMM algorithm}.
\renewcommand{\thealgorithm}{PD-PMM}

\begin{algorithm}[!ht]
\caption{Primal-dual proximal method of mutlipliers}
    \label{primal-dual PMM algorithm}
    \textbf{Input:}  $(x_0,y_0,z_0) \in \mathbb{R}^n \times \mathbb{R}^{m+l} \times \mathbb{R}^n$, $\beta_0,\ \beta_{\infty},\ \tau_{\infty} > 0$, $\{\tau_k\}_{k=0}^{\infty}$ such that $\tau_k \searrow \tau_{\infty} > 0$.
\begin{algorithmic}
\State Choose a sequence of positive numbers $\{\epsilon_k\}$ such that $\epsilon_k \rightarrow 0$. 
\For {($k = 0,1,2,\ldots$)}
\State Find $x_{k+1}$ such that:
\begin{equation} \label{primal-dual PMM main sub-problem}
\left\|\nabla \phi(x_{k+1})\right\|  \leq \epsilon_k.
\end{equation}
   \begin{flalign}  \label{primal-dual PMM y-update}
\ \ \ \ \ y_{k+1} &= \  \begin{bmatrix}
    y_{1_k} - \beta_k (A x_{k+1} - b)\\
    \Pi_{\mathcal{C}}\left(y_{2_k} + \beta_k (C x_{k+1}+ d) \right)
\end{bmatrix}.&&
\end{flalign} 
   \begin{flalign}  \label{primal-dual PMM z-update}
\ \ \ \ \ z_{k+1} &= \  (z_k + \beta_k x_{k+1}) - \beta_k\Pi_{\mathcal{K}}\big(\beta_k^{-1}z_k + x_{k+1} \big).&&
\end{flalign} 
\begin{flalign} \ \ \ \ \ \beta_{k+1} &\nearrow \beta_{\infty} \leq \infty, \quad \rho_{k+1} = \frac{\beta_{k+1}}{\tau_{k+1}}.&&
\end{flalign}
\EndFor
\State \Return $(x_k,y_k,z_k)$.
\end{algorithmic}
\end{algorithm}
\par Notice that we allow step \eqref{primal-dual PMM main sub-problem} to be computed inexactly. In Section \ref{subsection: PMM convergence analysis} we will provide precise conditions on the error sequence guaranteeing that Algorithm \ref{primal-dual PMM algorithm} achieves global convergence, and additional conditions for achieving a local linear or superlinear rate (where the local superlinear convergence requires that $\beta_k \rightarrow \infty$). Further conditions on the starting point and on the starting penalty parameter $\beta_0$, required to guarantee a global linear convergence rate, are also discussed.  

\subsection{Convergence analysis} \label{subsection: PMM convergence analysis}
\par In this section we provide conditions on the error sequence $\{\epsilon_k\}$ in \eqref{primal-dual PMM main sub-problem} that guarantee the convergence of Algorithm \ref{primal-dual PMM algorithm}, potentially at a global linear or local superlinear rate. The analysis is derived by applying the results shown in \cite[Section 2]{SIAMOpt:Lietal} or \cite{RockafellarGeneralizedPMM} (or by an extension of the analyses in \cite{MathOpRes:Rock,SIAMJCO:Rock}) after connecting Algorithm \ref{primal-dual PMM algorithm} to an appropriate proximal point iteration. As before, we let $F = [A^\top\ -C^\top]^\top$ and $\hat{b} = [b^\top\ d^\top]^\top$ and define the maximal monotone operator $T_{\ell} \colon \mathbb{R}^{2n+m+l} \rightrightarrows \mathbb{R}^{2n+m+l}$, associated to \eqref{primal problem}--\eqref{dual problem}:
\begin{equation*} \label{maximal monotone operator associated to primal-dual problem}
\begin{split}
T_{\ell}(x,y,z) &\triangleq \Bigg\{(u',v',w')\ \bigg\vert\  u' \in  \nabla f(x) - F^\top y + z,\  v' \in Fx -\hat{b} + \begin{bmatrix} 0_m\\\partial h^*(y_2) \end{bmatrix},\ w'+ x \in \partial\delta^*_{\mathcal{K}}(z)\Bigg\}.
\end{split}
\end{equation*}
\noindent  The inverse of this operator reads
\begin{equation*} \label{inverse of maximal monotone operator associated to primal-dual problem}
\begin{split}
T^{-1}_{\ell}(u',v',w') &\triangleq \arg\max_{y,z}\min_x\left\{\ell(x,y,z) + u'^\top x - v'^\top y -w'^\top z\right\}.
\end{split}
\end{equation*}
\noindent Notice that Assumption \ref{assumption: solution of the QP} implies that $T^{-1}_{\ell}(0) \neq \emptyset$. Following the result in \cite{SIAMOpt:Lietal}, and since $f$ is convex quadratic while $h$ is piecewise linear, we note that $T_{\ell}$ is in fact a polyhedral multifunction (see \cite{RobinsonMathProgStud} for a detailed discussion on the properties of such multifunctions). In light of this property of $T_{\ell}$ we note that the following specialized \emph{metric subregularity} condition (see \cite{Springer:DonRock} for a definition) holds automatically without any additional assumptions.
\begin{lemma}\label{lemma: error condition for polyhedral multifucntion}
For any $r > 0$, there exists $\kappa > 0$ such that
\begin{equation} \label{eqn: error condition for polyhedral multifunction}
\textnormal{dist}\big(p,T_{\ell}^{-1}(0)\big) \leq \kappa\ \textnormal{dist}\big(0,T_{\ell}(p)\big),\quad \forall\ p \in \mathbb{R}^{2n+m},\ \textnormal{with }\textnormal{dist}\big(p,T_{\ell}^{-1}(0)\big) \leq r.
\end{equation}
\end{lemma}
\begin{proof}
\noindent The reader is referred to \cite[Lemma 2.4]{SIAMOpt:Lietal} as well as \cite{RobinsonMathProgStud}.
\end{proof}
\par Next, let some sequence of positive definite matrices $\{R_k\}_{k=0}^{\infty}$ be given with $R_k \triangleq \tau_k I_n \oplus I_{m+l} \oplus I_n$, for all $k \geq 0$, where $\tau_k$ is defined in Algorithm \ref{primal-dual PMM algorithm} and $\oplus$ denotes the direct sum of two matrices. We define the single-valued proximal operator $P_k \colon \mathbb{R}^{2n + m + l} \mapsto \mathbb{R}^{2n+m + l}$, associated with $T_{\ell}$:
\begin{equation} \label{proximal operator}
P_k \triangleq \big(R_k + \beta_k T_{\ell}\big)^{-1} R_k.
\end{equation}
\noindent In particular, under our assumptions on the matrices $R_k$, we have that (e.g. see \cite{SIAMJCO:Rock}) for all $(u_1,v_1,w_1),\ (u_2,v_2,w_2) \in \mathbb{R}^{2n+m+l}$, the following inequality (non-expansiveness) holds
\begin{equation} \label{non-expansiveness of P_k}
\left\|(u_1,v_1,w_1)-P_k(u_2,v_2,w_2)\right\|_{R_k} \leq \left\|(u_1,v_1,w_1)-(u_2,v_2,w_2)\right\|_{R_k}. \end{equation}
\noindent Obviously, we can observe that if $(x^*,y^*,z^*) \in T_{\ell}^{-1}(0)$, then $P_k(x^*,y^*,z^*) = (x^*,y^*,z^*)$. We are now able to connect Algorithm \ref{primal-dual PMM algorithm} with the proximal point iteration produced by \eqref{proximal operator}.
\begin{Proposition} \label{proposition: connection to PMM}
Let $\{(x_k,y_k,z_k)\}_{k=0}^{\infty}$ be a sequence of iterates produced by Algorithm \textnormal{\ref{primal-dual PMM algorithm}}. Then, for every $k \geq 0$ we have that
\begin{equation} \label{connection of PMM with PPA error}
\left\|(x_{k+1},y_{k+1},z_{k+1})-P_k(x_k,y_k,z_k)\right\|_{R_k} \leq \frac{\beta_k}{\min\{\sqrt{\tau_k},1\}}\left\|\nabla \phi(x_{k+1})\right\|.
\end{equation}
\end{Proposition}
\begin{proof}
\noindent From \eqref{eqn: proximal Lagrangian penalty}, we observe that 
$$\nabla \phi(x_{k+1}) = \nabla \mathcal{L}_{\beta_k}(x_{k+1};y_k,z_k) + \frac{1}{2\rho_k}(x_{k+1}-x_k).$$
\noindent Thus, by utilizing \cite[Proposition 7]{MathOpRes:Rock}, we deduce
\begin{equation*}\label{connection of PPM with PPA eqn}
\left(\nabla \phi(x_{k+1}),0_{m+l},0_n\right)  +
\beta_k^{-1}\left( \tau_k(x_k - x_{k+1}),y_k - y_{k+1},z_k - z_{k+1}\right) \in\ T_{\ell}(x_{k+1},y_{k+1},z_{k+1}).
\end{equation*}
\noindent Subtracting $P_k(x_k,y_k,z_k)$ from both sides, taking norms, and using the non-expansiveness of $P_k$ (see \eqref{non-expansiveness of P_k}), yields \eqref{connection of PMM with PPA error} and concludes the proof.
\end{proof}
\par Now that we have established the connection of Algorithm \ref{primal-dual PMM algorithm} with the proximal point iteration governed by the operator $P_k$ defined in \eqref{proximal operator}, we can directly provide conditions on the error sequence in \eqref{primal-dual PMM main sub-problem}, to guarantee global (possibly linear) and local linear (potentially superlinear) convergence of Algorithm \ref{primal-dual PMM algorithm}. To that end, we will make use of certain results, as reported in \cite[Section 2]{SIAMOpt:Lietal}. Firstly, we provide the global convergence result for the algorithm.
\begin{theorem}
Let Assumption \textnormal{\ref{assumption: solution of the QP}} hold. Let $\{(x_k,y_k,z_k)\}_{k=0}^{\infty}$ be generated by Algorithm \textnormal{\ref{primal-dual PMM algorithm}}. Furthermore, assume that we choose a sequence $\{\epsilon_k\}_{k=0}^{\infty}$ in \eqref{primal-dual PMM main sub-problem}, such that
\begin{equation} \label{eqn: error condition for global convergence}
\epsilon_k \leq \frac{\min\{\sqrt{\tau_k},1\}}{\beta_k}\delta_k,\quad 0 \leq \delta_k,\quad \sum_{k=0}^{\infty} \delta_k < \infty.
\end{equation}
\noindent Then, $\{(x_k,y_k,z_k)\}_{k=0}^{\infty}$ is bounded and converges to a primal-dual solution of \eqref{primal problem}--\eqref{dual problem}.
\end{theorem}
\begin{proof}
\noindent The proof is omitted since it is a direct application of \cite[Theorem 2.3]{SIAMOpt:Lietal}. 
\end{proof}
\par Next, we discuss local linear (and potentially superlinear) convergence of Algorithm \ref{primal-dual PMM algorithm}. To that end, let $r > \sum_{k = 0}^{\infty} \delta_k$, where $\delta_k$ is defined in \eqref{eqn: error condition for global convergence}. Then, from Lemma \ref{lemma: error condition for polyhedral multifucntion} we know that there exists $\kappa > 0$ associated with $r$ such that
\begin{equation} \label{eqn: specific kappa and r for error condition}
 \textnormal{dist}\big((x,y,z),T_{\ell}^{-1}(0)\big) \leq \kappa\ \textnormal{dist}\big(0,T_{\ell}(x,y,z)\big),
 \end{equation}
\noindent for all $(x,y,z) \in \mathbb{R}^{2n+m}$ such that $\textnormal{dist}\big((x,y,z),T_{\ell}^{-1}(0)\big) \leq r$. 
\begin{theorem} \label{thm:local convergence} Let Assumption \textnormal{\ref{assumption: solution of the QP}} hold, and assume that $(x_0,y_0,z_0)$ is chosen to satisfy \[\textnormal{dist}_{R_0}\big((x_0,y_0,z_0),T_{\ell}^{-1}(0)\big) \leq r - \sum_{k = 0}^{\infty} \delta_k,\]
\noindent where $\{\delta_k\}_{k=0}^{\infty}$ is given in \eqref{eqn: error condition for global convergence}. Let also $\kappa$ be given as in \eqref{eqn: specific kappa and r for error condition} and assume that we choose a sequence $\{\epsilon_k\}_{k=0}^{\infty}$ in \eqref{primal-dual PMM main sub-problem} such that
\begin{equation} \label{eqn: error condition for superlinear convergence}
\epsilon_k \leq \frac{\min\{\sqrt{\tau_k},1\}}{\beta_k} \min\big\{\delta_k, \delta_k' \|(x_{k+1},y_{k+1},z_{k+1}) - (x_k,y_k,z_k)\|_{R_k}\big\},
\end{equation}
\noindent where $0 \leq \delta_k$, $\sum_{k=0}^{\infty} \delta_k < \infty$, and $0 \leq \delta_k' < 1$, $\sum_{k = 0}^{\infty} \delta_k' < \infty$. Then, for all $k \geq 0$ we have that
\begin{equation*} \label{eqn: superlinear convergence error decay}
\begin{split}
\textnormal{dist}_{R_{k+1}}\left((x_{k+1},y_{k+1},z_{k+1},T_{\ell}^{-1}(0)\right) & \leq \mu_k \textnormal{dist}_{R_k}\left( x_{k},y_{k},z_{k},T_{\ell}^{-1}(0)\right),
\end{split}
\end{equation*}
\[\mu_k \triangleq (1-\delta_k')^{-1}\left(\delta_k' + (1+\delta_k')\frac{\kappa \gamma_k}{\sqrt{\beta_k^2 + \kappa^2\gamma_k^2}}\right), \qquad \lim_{k\rightarrow \infty} \mu_k = \mu_{\infty} \triangleq \frac{\kappa \gamma_{\infty}}{\sqrt{\beta_{\infty}^2 + \kappa^2\gamma_{\infty}^2}},\]
\noindent where $\gamma_k \triangleq \max\{\tau_k,1\}$ and $\gamma_{\infty} = \max\{\tau_{\infty},1\}$ (noting that $\mu_{\infty} = 0,\ \textnormal{if }\beta_{\infty} = \infty$).
\end{theorem}
\begin{proof}
\noindent The proof is omitted since it follows by direct application of \cite[Theorem 2.5]{SIAMOpt:Lietal} (see also \cite[Theorem 2]{SIAMJCO:Rock}).
\end{proof}
\begin{remark}
Following \textnormal{\cite[Remarks 2, 3]{SIAMOpt:Lietal}}, we can choose a non-increasing sequence $\{\delta_k'\}_{k=0}^{\infty}$ and a large enough $\beta_0$ such that $\mu_0 < 1$, which in turn implies that $\mu_k \leq \mu_0 < 1$, yielding a global linear convergence of both $\textnormal{dist}\big((x_{k},y_k,z_k),T_{\ell}^{-1}(0)\big)$ as well as $\textnormal{dist}_{R_k}\big((x_{k},y_k,z_k),T_{\ell}^{-1}(0)\big)$, assuming that the starting point of the algorithm satisfies the assumption stated in Theorem \textnormal{\ref{thm:local convergence}}. On the other hand, as is implicitly mentioned in Theorem \textnormal{\ref{thm:local convergence}}, if $\beta_k$ is forced to increase indefinitely, we obtain a local superlinear convergence rate (notice that $\mu_{\infty} = 0$ if $\beta_{\infty} = \infty$).
\end{remark}

\section{Semismooth Newton method} \label{sec: SSN method}
\par In this section we briefly present a standard semismooth Newton (SSN) scheme suitable for the solution of problem \eqref{primal-dual PMM main sub-problem}, appearing in Algorithm \ref{primal-dual PMM algorithm}. 
To that end, we fix the estimates $(x_k,y_k,z_k)$ as well as the penalty parameters $\beta_k,\ \rho_k$. We set $x_{k_0} = x_k$, and at every iteration $j$ of SSN, we approximately solve a system of the following form:
\begin{equation} \label{Primal-dual SSN system}
J_{k_j}
d_x= -\nabla \phi(x_{k_j}),
\end{equation}
\noindent where $J_{k_j} \in \partial^C_x\left(\nabla \phi(x_{k_j})\right)$. The symbol $\partial_x^C(\cdot)$ denotes the \emph{Clarke subdifferential} of a function (see \cite{JWS:Clarke}) with respect to $x$, which can be obtained as the convex hull of the \emph{Bouligand subdifferential} (\cite{JWS:Clarke}). Any element of the Clarke subdifferential is a \emph{Newton derivative} (see \cite[Chapter 13]{arXiv:ClasValk}), since $\nabla \phi(x_{k_j})$ is a \emph{piecewise continuously differentiable} and \emph{regular function}. Using \cite[Theorem 14.8]{arXiv:ClasValk}, we obtain that for any interval $[l,u]$ on the real line (i.e., $l < u$):
\begin{equation*}
\partial_{w}^C \left(\Pi_{[l,u]}(w_i)\right)= \begin{cases} 
\{1\}, &\qquad \textnormal{if}\quad w \in (l,u),\\
\{0\}, &\qquad \textnormal{if}\quad w \notin [l,u],\\
[0,1], &\qquad \textnormal{if}\quad w \in \{l,u\}.
\end{cases}
\end{equation*}
\par For computational reasons, we always choose matrices $M_{k_j}$ from the Bouligand subdifferential. This choice will become apparent in the Section \ref{sec: solution of linear systems}. On the other hand, it is well-known (see \cite[Theorem 4]{CAM:MartQi}) that an inexact semismooth Newton scheme using the Bouligand subdifferential converges at a local linear rate (assuming that the linear systems are solved up to an appropriate accuracy), if the equation $\nabla \phi(x) = 0_n$ is \emph{BD-regular} at the optimum $(x_k^*)$ (that is, each element of the Bouligand subdifferential of $\nabla\phi(x_k^*)$ is nonsingular). Note, however, that since we employ the semismooth Newton scheme to solve the sub-problems arising from Algorithm \ref{primal-dual PMM algorithm}, we obtain that the resulting nonsmooth equations are indeed BD-regular for every outer iteration $k \geq 0$. Thus, assuming that the associated linear systems are solved up to a sufficient accuracy (see \cite[Theorem 4]{CAM:MartQi}), we obtain local linear convergence rate of the resulting inexact semismooth Newton scheme. If, additionally, the solution accuracy of the associated linear systems is increased at a suitable rate, the resulting local rate can be superlinear. 
\par We complete the derivation of the SSN by introducing backtracking line-search to ensure global convergence. Then, under no additional assumptions, one can show that SSN is globally convergent.  Algorithm \ref{primal-dual SNM algorithm} outlines a semismooth Newton method for the solution of \eqref{primal-dual PMM main sub-problem}. We assume that the associated linear systems are approximately solved by means of a Krylov subspace method. An analysis of the effect of errors arising from the use of Krylov methods within SSN applied to nonsmooth equations can be found in \cite{NLAA:ChenQi}. 
\renewcommand{\thealgorithm}{SSN}
\begin{algorithm}[!ht]
\caption{Semismooth Newton method}
    \label{primal-dual SNM algorithm}
    \textbf{Input:} $\epsilon_k > 0$, $\mu \in \left(0,\frac{1}{2}\right)$, $\delta \in (0,1)$, $\gamma \in (0,1]$, $\bar{\eta} \in (0,1)$, $x_{k_0} = x_k$,  $y_{k_0} = y_k$.
\begin{algorithmic}
\For {($j = 0,1,2,\ldots$)}
\State Choose $J_{k_j} \in \partial^C_x\left(\nabla \phi(x_{k_j})\right)$, and solve
\[J_{k_j}
d_x= -\nabla \phi(x_{k_j}),\]
\State such that $\left\| J_{k_j} d_x + \nabla \phi(x_{k_j})\right\| \leq   \min\left\{\bar{\eta},\left\|\nabla\phi(x_{k_j})\right\|^{1+\gamma}\right\}.$ 
\State (Line-search) Set $\alpha_j = \delta^{m_j}$, where $m_j$ is the first non-negative integer for which:
\[\phi\left(x_{k_j} + \delta^{m_j}d_x\right) \leq  \phi(x_{k_j}) + \mu \delta^{m_j}\left(\nabla \phi(x_{k_j}) \right)^\top d_x\]
\State $x_{k_{j+1}} = x_{k_j} + \alpha_j d_x$.
\If {$\left(\left\|\nabla \phi(x_{k_{j+1}})\right\| \leq \epsilon_k\right)$}
\State \Return $x_{k_{j+1}}$.
\EndIf
\EndFor
\end{algorithmic}
\end{algorithm}
\par Below we provide a theorem characterizing the convergence of Algorithm \ref{primal-dual SNM algorithm}.
\begin{theorem} \label{thm: SSN convergence}
    Let $\{x_{k_j}\}_{j = 0}^{\infty}$ be the sequence of iterates generated by Algorithm \textnormal{\ref{primal-dual SNM algorithm}}. Then, this sequence converges to the unique optimal solution $x_k^*$ of $\min_{x} \phi(x)$, and $\|x_{k_{j+1}} - x_k^*\| = \mathcal{O}\left(\|x_{k_{j}} - x_k^*\|^{1+\gamma}\right)$. 
\end{theorem}
\begin{proof}
    The proof is omitted since it follows by a trivial extension of the results given in \cite[Proposition 3.3, Theorems 3.4-3.5]{SIAMOpt:ZhaoSunToh}.
\end{proof}
\section{The SSN linear systems} \label{sec: solution of linear systems}
\par The major bottleneck of the proposed active-set scheme is the approximate solution of the linear systems given in \eqref{Primal-dual SSN system}. Since Algorithm \ref{primal-dual SNM algorithm} does not require an exact solution, we can utilize preconditioned Krylov subspace solvers for the efficient solution of such systems. In what follows we will derive an equivalent reformulation of \eqref{Primal-dual SSN system}, and we will subsequently propose and analyze a robust preconditioner for the acceleration of the associated Krylov subspace method employed for its solution.
\subsection{Derivation of the linear systems}
\par Before we proceed with the derivation of the linear systems, we first need to note that our implementation treats separable nonsmooth terms differently for improved memory and computational efficiency. Specifically, if the problem given in \eqref{primal problem} contains a term $\sum_{i=1}^n \left(\left(Dx+d\right)_i\right)_+$, where $D$ is a diagonal (``weight") matrix, we will discuss why and how this is treated differently. To that end, we assume that the matrix $C$ appearing in \eqref{primal problem} has the form $C = [\hat{C}^\top\ D]^\top$, where $\hat{C} \in \mathbb{R}^{s \times n}$ is some rectangular matrix, and $D \in \mathbb{R}^{n \times n}$ is a diagonal square matrix, with $s+n = l$. While this structure might seem rather specific, it is very common as it appears when dealing with sparse approximation problems, in which the objective function contains a term of the form $\|Dx\|_1$.
\par Let $k,\ j \geq 0$ be some arbitrary iterations of Algorithm \ref{primal-dual PMM algorithm}, and \ref{primal-dual SNM algorithm}, respectively. Given the assumed structure of matrix $C$, let us re-write the gradient of \eqref{eqn: proximal Lagrangian penalty}:
\begin{equation*}
\begin{split}
\nabla \phi(x) &= \nabla f(x) - A^\top y_{1_k} + \beta_kA^\top(Ax-b) + \hat{C}^\top \Pi_{\mathcal{C}_s}\left(\left(y_{2_k}\right)_{1:s} + \beta_k\left( \hat{C}x+(d)_{1:s}\right) \right)\\
&\qquad + D \Pi_{\mathcal{C}_n}\left(\left(y_{2_k}\right)_{s+1:l} + \beta_k\left( Dx+(d)_{s+1:l}\right) \right)\\
&\qquad  + (z_k + \beta_k x) - \beta_k \Pi_{\mathcal{K}}(\beta_k^{-1}z_k + x) + \rho_k^{-1}(x-x_k),
\end{split}
\end{equation*}
\noindent where $\mathcal{C}_{q} = \left\{u \in \mathbb{R}^q \big\vert u_i \in [0,1],\ i \in \{1,\ldots,q\} \right\}.$ We notice that any element 
\[B_{\delta,k_j} \in \partial_w^C\left(\Pi_{\mathcal{K}}\left(w\right)\right)\big\vert_{w =\beta_k^{-1}z_{k} + x_{k_j}} \]
\noindent yields a Newton derivative (see \cite[Theorem 14.8]{arXiv:ClasValk}). The same applies for any \[{B}_{h,k_j} \in \partial_w^C\left(\Pi_{\mathcal{C}}\left(w\right)\right)\big\vert_{w =y_{2_k}+\beta_kC(x_{k_j}+d) }.\]
\noindent We choose $B_{\delta,k_j},\ {B}_{h,k_j}$ from the Bouligand subdifferential to improve computational efficiency. We set $B_{\delta,k_j} \in \mathbb{R}^{n\times n}$ as the diagonal matrix with
\begin{equation}\label{eqn: Clarke subdifferential of projection choice}
\begin{split}
\left(B_{\delta,k_j}\right)_{i,i} \triangleq & \begin{cases} 
1, &\quad \textnormal{if}\  \beta_k^{-1}z_{k,i} + x_{k_j,i} \in (a_{l_i},a_{u_i}),\\
0, &\quad  \textnormal{otherwise}, 
\end{cases}
\end{split}
\end{equation}
\noindent for $i \in \{1,\ldots,n\}$. On the other hand, we separate two cases for $B_{h,k_j} \in \mathbb{R}^{l\times l}$. Specifically, we set
\begin{equation} \label{eqn: Clark subdiff of h projectors}
\begin{split}
\left({B}_{h,k_j}\right)_{i,i} \triangleq & \begin{cases} 
1, &\quad \textnormal{if}\  \left(Cx + d\right)_i \in (0,1),\\
0, &\quad  \textnormal{otherwise},
\end{cases},\quad \text{for } i \in \{1,\ldots,s\},\ \text{and } \\
\left({B}_{h,k_j}\right)_{i,i} \triangleq & \begin{cases} 
1, &\quad \textnormal{if}\  \left(Cx + d\right)_i \in [0,1],\\
0, &\quad  \textnormal{otherwise},
\end{cases}, \quad \text{for } i \in \{s+1,\ldots,l\}.
\end{split}
\end{equation}
\noindent Having defined these, we can now formulate the chosen Newton derivative of $\nabla \phi(x_{k_j})$, as
\[ M_{k_j} = Q - \beta_k A^\top A + \beta_k C^\top B_{h,k_j} C + \beta_k I_n - \beta_k B_{\delta,k_j} + \rho_k^{-1} I_n.\]
\noindent However, by letting $B^1_{h,k_j} \triangleq (B_{h,k_j})_{1:s,1:s}$ and  $B^2_{h,k_j} \triangleq (B_{h,k_j})_{s+1:l,s+1:l}$, and by introducing two auxiliary variables $w_1,\ w_2$, we observe that the semismoooth Newton system in \eqref{Primal-dual SSN system} can equivalently be written as 
\begin{equation} \label{eqn: equivalent SSN linear system}
    \begin{split}
        & \begin{bmatrix}
          -(Q + \rho_k^{-1}I_n + \beta_k(I_n - B_{\delta,k_j}) + \beta_k D B^2_{h,k_j} D) & A^\top & -\hat{C}^\top\\
          A & \beta_k^{-1}I_m & 0_{m,s}\\
          -B^1_{h,k_j} \hat{C} & 0_{s,m} & \beta_k^{-1} I_s
        \end{bmatrix} \begin{bmatrix}
            d_x\\
            w_1\\
            w_2
        \end{bmatrix}  = \begin{bmatrix} \xi_1 \\ \xi_2 \\ \xi_3 \end{bmatrix},
        \end{split}
        \end{equation}
    \noindent where
        \begin{equation*} \begin{bmatrix} \xi_1 \\ \xi_2 \\ \xi_3 \end{bmatrix} \triangleq \begin{bmatrix}
    c + Qx_{k_j} + D \Pi_{\mathcal{C}_n}\left((y_{2_k})_{s+1:l} + \beta_k (Dx_{k_j} + (d)_{s+1:l} )\right) + \tilde{z}_{k_j} + \rho_k^{-1}(x_{k_j}-x_k)\\
    \beta_k^{-1}y_{1_k}-(Ax_{k_j}-b)\\
    \beta_k^{-1}\Pi_{\mathcal{C}_s}\left(\left(y_{2_k} + \beta_k(Cx_{k_j} + d)\right)_{1:s}\right)
\end{bmatrix},
\end{equation*}
\noindent with $\tilde{z}_{k_j} \triangleq (z_k +\beta_k x_{k_j}) - \beta_k \Pi_{\mathcal{K}}\left(\beta_k^{-1}z_k + x_{k_j} \right).$ Derived from \eqref{eqn: Clark subdiff of h projectors}, we define four index sets: 
\[{\mathcal{B}}^1_{h,j} \triangleq \left \{ i \in \{1,\ldots,s\} \bigg\vert \left({B}^1_{h,k_j}\right)_{i,i} = 1 \right\},\qquad  {\mathcal{N}}^1_{h,j} \triangleq \{1,\ldots,s\}\setminus {\mathcal{B}}^1_{h,j},\]
\noindent and 
\[{\mathcal{B}}^2_{h,j} \triangleq \left \{ i \in \{1,\ldots,n\} \bigg\vert \left({B}^2_{h,k_j}\right)_{i,i} = 1 \right\},\qquad  {\mathcal{N}}^2_{h,j} \triangleq \{1,\ldots,s\}\setminus {\mathcal{B}}^2_{h,j}.\]
\noindent  Observe that 
$$B^1_{h,k_j}\hat{C} = \mathscr{P}\begin{bmatrix}
    \left(\hat{C}\right)_{\mathcal{B}^1_{h,k_j},1:n}\\
    0_{\left|\mathcal{N}^1_{h,k_j}\right|,n}
\end{bmatrix},$$
\noindent where $\mathscr{P}$ is an appropriate permutation matrix. Then, from \eqref{eqn: equivalent SSN linear system} we obtain
\[ \left(w_2\right)_{\mathcal{N}^1_{h,k_j}} = \left(\Pi_{\mathcal{C}_s}\left(\left(y_{2_k} + \beta_k(Cx + d)\right)_{1:s}\right) \right)_{\mathcal{N}^1_{h,k_j}},\]
\noindent and \eqref{eqn: equivalent SSN linear system} reduces to
\begin{equation} \label{eqn: reduced equivalent SSN linear system}
    \begin{split}
        &\underbrace{\begin{bmatrix}
          -(Q + \rho_k^{-1}I_n + \beta_k(I_n - B_{\delta,k_j}) + \beta_k D B^2_{h,k_j} D) & A^\top & -\left(\hat{C}^\top\right)_{\mathcal{B}^1_{h,k_j}}\\
          A & \beta_k^{-1}I_m & 0_{m,|\mathcal{B}^1_{h,k_j}|}\\
          -\left(\hat{C}\right)_{\mathcal{B}^1_{h,k_j},1:n} & 0_{|\mathcal{B}^1_{h,k_j} |,m} & \beta_k^{-1} I_{|\mathcal{B}^1_{h,k_j}|}
        \end{bmatrix}}_{{M}_{k_j}} \begin{bmatrix}
            d_x\\
            w_1\\
            (w_2)_{\mathcal{B}^1_{h,k_j}}
        \end{bmatrix}\\
        &\qquad = \begin{bmatrix} \xi_1 \\ \xi_2 \\ (\xi_3)_{\mathcal{B}_{h,k_j}^1} \end{bmatrix} + \begin{bmatrix}
    \left(\hat{C}^\top\right)_{\mathcal{N}^1_{h,k_j}} (w_2)_{\mathcal{N}^1_{h,k_j}}\\
   0_m\\
    0_{|\mathcal{B}^1_{h,k_j}|}
\end{bmatrix},
    \end{split}
\end{equation}
\noindent the coefficient matrix of which is \emph{quasi-definite} (see \cite{SIAMOpt:Vander}), invertible, and its conditioning can be directly controlled by the regularization parameters $(\beta_k,\rho_k)$ of Algorithm \ref{primal-dual PMM algorithm}.
\subsection{Preconditioning and iterative solution via Krylov subspace methods}
\par Next, we would like to construct an effective preconditioner for ${M}_{k_j}$ (i.e., the coefficient matrix of \eqref{eqn: reduced equivalent SSN linear system}). Before delving into this, let us observe that in the derivation of this linear system, we utilized the active-set $\mathcal{B}_{h,k_j}^1$ to reduce the size of the resulting coefficient matrix. As before, let us define
\[ \mathcal{B}_{\delta,k_j} \triangleq \left\{i \in \{1,\ldots,n\} \big\vert \left(B_{\delta,k_j}\right)_{i,i} = 1 \right\},\qquad \mathcal{N}_{\delta,k_j} \triangleq \{1,\ldots,n\} \setminus \mathcal{B}_{\delta,k_j}.\]
\noindent In what follows, we will make use of the active sets $\mathcal{B}^2_{h,k_j}$ and $\mathcal{B}_{\delta,k_j}$ to derive cheap and effective preconditioners for the iterative solution of \eqref{eqn: reduced equivalent SSN linear system} by means of Krylov subspace solvers. For simplicity of notation, let us define
\[ H_{k_j} \triangleq Q + \rho_k^{-1}I_n + \beta_k(I_n - B_{\delta,k_j}) + \beta_k D B^2_{h,k_j} D, \]
\noindent and consider an approximation of this matrix that reads
\[ \widetilde{H}_{k_j} \triangleq \textnormal{Diag}(Q) + \rho_k^{-1}I_n + \beta_k(I_n - B_{\delta,k_j}) + \beta_k D B^2_{h,k_j} D.\]
\noindent Let us also define the compound matrix 
\[ G \triangleq \begin{bmatrix}
    A\\ -\left(\hat{C}\right)_{\mathcal{B}^1_{h,k_j},1:n}
\end{bmatrix}, \]
\noindent where we drop the iteration dependence for simplicity, and two sets of interest 
\[\mathcal{B}_{k_j} \triangleq  \mathcal{B}_{\delta,k_j} \cap \mathcal{N}_{h,k_j}^2,\qquad \mathcal{N}_{k_j} \triangleq \{1,\ldots,n\}\setminus \mathcal{B}_{k_j}. \]
\noindent We are now ready to define a positive definite preconditioner for the coefficient matrix in \eqref{eqn: reduced equivalent SSN linear system}, as follows:
\begin{equation} \label{eqn: preconditioner for explicit SSN linear system}
{P}_{k_{j}} \triangleq \begin{bmatrix}
 \widetilde{H}_{k_{j}} & 0_{n,m+|\mathcal{B}^1_{h,k_j} |} \\
0_{m+|\mathcal{B}^1_{h,k_j} |,n} &  \left(G E_{k_{j}} G^\top + \beta_k^{-1} I_{m+|\mathcal{B}^1_{h,k_j} |}\right)
\end{bmatrix},
\end{equation}
\noindent with $E_{k_{j}} \in \mathbb{R}^{n \times n}$ the diagonal matrix defined as
\begin{equation} \label{eqn: preconditioner dropping matrix E}
\left(E_{k_{j}}\right)_{(i,i)} \triangleq  \begin{cases} 
\frac{1}{\left(\widetilde{H}_{k_{j}}\right)_{i,i}}, &\qquad \textnormal{if}\quad i \in \mathcal{B}_{k_j},\\
0, &\qquad  \textnormal{otherwise}.
\end{cases}
\end{equation}
\noindent Thus, we utilize a block diagonal preconditioner in which we drop all columns of $G$ that belong to $\mathcal{N}_{k_j}$. This provides a clear motivation behind our choice of matrices in \eqref{eqn: Clarke subdifferential of projection choice} and \eqref{eqn: Clark subdiff of h projectors}.
\par The preconditioner in \eqref{eqn: preconditioner for explicit SSN linear system} is an extension of the preconditioner proposed in \cite{NLAA:BergGondMartPearPoug} for the solution of linear systems arising from the application of a regularized interior point method to convex quadratic programming. Being a diagonal matrix, $E_{k_{j}}$ yields a sparse approximation of the Schur complement of the saddle-point matrix in \eqref{eqn: preconditioner for explicit SSN linear system}. This approximation is then used to construct a positive definite block-diagonal preconditioner (i.e. ${P}_{k_{j}}$), which can be used within a symmetric Krylov solver, like the minimum residual (MINRES) method (see \cite{PaigeSaundersSIAMNumAnal}). The $(2,2)$ block of ${P}_{k_j}$ can be inverted via a Cholesky decomposition.
\par Next, we analyze the spectral properties of the preconditioned matrix $({P}_{k_{j}})^{-1}{M}_{k_j}$. Let 
\[{S_P}_{k_j} \triangleq \left(G  \widetilde{H}_{k_j}^{-1} G^\top + \beta_k^{-1} I_m\right),\qquad {S}_{k_j} \triangleq \left(G E_{k_{j}} G^\top + \beta_k^{-1} I_m\right). \]
\noindent In the following lemma, we bound the eigenvalues of the preconditioned matrix ${S_P}_{k_j}^{-1}{S}_{k_j}$. This is subsequently used to analyze the spectrum of ${P}_{k_j}^{-1}{M}_{k_j}$.
\begin{lemma} \label{lemma: spectral properties of approximate preconditioned Schur complement}
For any iterates $k$ and $j$ of Algorithms \textnormal{\ref{primal-dual PMM algorithm}} and \textnormal{\ref{primal-dual SNM algorithm}}, respectively, we have
\[1 \leq \lambda \leq  1+\sigma_{\max}^2(\hat{F})\left(\frac{1}{2+\beta^{-2}_{\infty}\tau_{\infty}}\right),\]
\noindent where $\lambda \in \lambda\left( {S_P}_{k_j}^{-1}{S}_{k_j}\right),$ $\hat{F} = [A^\top\ -\hat{C}^\top]^\top$, and $\beta_{\infty},\ \tau_{\infty}$ are defined in Algorithm \textnormal{\ref{primal-dual PMM algorithm}}.
\end{lemma}
\begin{proof}
\noindent Consider the preconditioned matrix ${S_P}_{k_j}^{-1}{S}_{k_j}$, and let $(\lambda,u)$ be its eigenpair. Then, $\lambda$ must satisfy the following equation:
\[\lambda = \frac{u^\top \left( G_{\mathcal{B}_{k_j}} D_{\mathcal{B}_{k_j}} G_{\mathcal{B}_{k_j}}^\top + \beta_k^{-1} I_m +  G_{\mathcal{N}_{k_j}} D_{\mathcal{N}_{k_j}} G_{\mathcal{N}_{k_j}}^\top\right) u}{u^\top \left( G_{\mathcal{B}_{k_j}} D_{\mathcal{B}_{k_j}} G_{\mathcal{B}_{k_j}}^\top + \beta_k^{-1 }I_m\right)u}, \]
\noindent where $D_{\mathcal{B}_{k_j}} = \left(\widetilde{H}^{-1}_{k_j}\right)_{\left(\mathcal{B}_{k_j},\mathcal{B}_{k_j}\right)}$, $D_{\mathcal{N}} = \left(\widetilde{H}^{-1}_{k_j}\right)_{\left(\mathcal{N}_{k_j},\mathcal{N}_{k_j}\right)}$. The above equality holds since $\widetilde{H}_{k_j}$ is a diagonal matrix, and $\left(E_{k_j}\right)_{(i,i)} = 0$ for every $i \in \mathcal{N}_{k_j}$ (indeed, see the definition in \eqref{eqn: preconditioner dropping matrix E}). Hence, from positive semi-definiteness of $Q$, we obtain
\begin{equation*}
\begin{split}
 1 &\ \leq \lambda = 1 + \frac{u^\top \left( G_{\mathcal{N}_{k_j}} D_{\mathcal{N}_{k_j}} G_{\mathcal{N}_{k_j}}\right) u}{u^\top \left( G_{\mathcal{B}_{k_j}} D_{\mathcal{B}_{k_j}} G_{\mathcal{B}_{k_j}}^\top + \beta_k^{-1 }I_m \right) u}  \leq 1 + \beta_k\sigma_{\max}^2(G_{\mathcal{N}_{k_j}})\left(2\beta_k+\rho_k^{-1}\right)^{-1} \\
 &\ \leq 1 + \sigma_{\max}^2(\hat{F})\left(\frac{1}{2+\beta_k^{-2}\tau_k}\right) \leq  1+\sigma_{\max}^2(\hat{F})\left(\frac{1}{2+\beta_{\infty}^{-2}\tau_{\infty}}\right),
\end{split}
\end{equation*}
\noindent where we used that $\rho_k = \beta_k/\tau_k$, $\beta_k \leq \beta_{\infty}$, and $\tau_k \geq \tau_{\infty}$. 
\end{proof}
\par Given Lemma \ref{lemma: spectral properties of approximate preconditioned Schur complement}, we are now able to invoke \cite[Theorem 3]{NLAA:BergGondMartPearPoug} to characterize the spectral properties of the preconditioned matrix ${P}_{k_j}^{-1}M_{k_j}$. Let 
\begin{equation*}
\begin{split}
\bar{S}_{k_j} \triangleq & \left({S_P}_{k_j}\right)^{-\frac{1}{2}}{S}_{k_j}\left({S_P}_{k_j}\right)^{-\frac{1}{2}},\ \qquad \qquad 
\bar{H}_{k_j} \triangleq \widetilde{H}_{k_j}^{-1/2} H_{k_j} \widetilde{H}_{k_j}^{-1/2},\\
\alpha_{NE} \triangleq &\ \lambda_{\min}(\bar{S}_{k_j}),\quad  \beta_{NE} \triangleq \lambda_{\max}(\bar{S}_{k_j}),\quad  \alpha_{H} \triangleq  \lambda_{\min}(\bar{H}_{k_j}),\quad \beta_{H} \triangleq \lambda_{\max}(\bar{H}_{k_j}).
\end{split}
\end{equation*}
\noindent Notice that Lemma \ref{lemma: spectral properties of approximate preconditioned Schur complement} yields upper and lower bounds for $\alpha_{NE}$ and $\beta_{NE}$. From the definition of $\bar{H}_{k_j}$ we can also obtain that $\alpha_{H} \leq 1 \leq \beta_{H}$ (see \cite{NLAA:BergGondMartPearPoug}). We are now ready to state the spectral properties of the preconditioned matrix  ${P}_{k_j}^{-1}{M}_{k_j}$.
\begin{theorem} \label{thm: spectral properties of preconditioned matrix}
Let $k$ and $j$ be some arbitrary iterates of Algorithms \textnormal{\ref{primal-dual PMM algorithm} and \ref{primal-dual SNM algorithm}}, respectively. Then, the eigenvalues of ${P}_{k_j}^{-1}{M}_{k_j}$ lie in the union of the following intervals:
\[I_{-} \triangleq \left[-\beta_{H} -\sqrt{\beta_{NE}}, -\alpha_{H}\right],\qquad I_+ \triangleq \left[\frac{1}{1+\beta_{H}},1+ \sqrt{\beta_{NE}-1}\right].\]
\end{theorem}
\begin{proof}
\noindent The proof follows by direct application of \cite[Theorem 3]{NLAA:BergGondMartPearPoug}.
\end{proof}
\begin{remark}
By combining Lemma \textnormal{\ref{lemma: spectral properties of approximate preconditioned Schur complement}} with Theorem \textnormal{\ref{thm: spectral properties of preconditioned matrix}}, we can observe that the eigenvalues of the preconditioned matrix ${P}_{k_j}^{-1} {M}_{k_j}$ are not deteriorating as $\beta_k \rightarrow \infty$. In other words, the preconditioner is robust with respect to the penalty parameters $\beta_k,\ \rho_k$ of Algorithm \textnormal{\ref{primal-dual PMM algorithm}}. Furthermore, our choices of $B_{h,k_j},\ {B}_{\delta,k_j}$ in \eqref{eqn: Clarke subdifferential of projection choice} and \eqref{eqn: Clark subdiff of h projectors} serve the purpose of further sparsifying the preconditioner in \eqref{eqn: preconditioner for explicit SSN linear system}, thus potentially further sparsifying its Cholesky decomposition. 
\par We note that the preconditioner in \textnormal{\eqref{eqn: preconditioner for explicit SSN linear system}} is a very efficient choice if $n \geq m+s$, which is the case in several of the experiments considered in Section \textnormal{\ref{sec: applications}}. However, any of the preconditioners given in \textnormal{\cite{arXiv:GondPougkPears}} can be directly applied for systems appearing in the proposed inner-outer scheme. For example, in the case where $m+s \geq n$ one could adapt the preconditioner given in \textnormal{\cite[Section 3.2]{arXiv:GondPougkPears}} (using the developments of this section), which would (potentially) be a more efficient choice when compared to that given in \textnormal{\eqref{eqn: preconditioner for explicit SSN linear system}}. Additionally, we should mention that for problems considered within this work, a diagonal approximation of the Hessian (within the preconditioner) seems sufficient to deliver very good performance. Indeed, this is the case for a wide range of problems. However, in certain instances, one might consider non-diagonal approximations of the Hessian. In that case, the preconditioner in \textnormal{\eqref{eqn: preconditioner for explicit SSN linear system}} can be readily generalized and analyzed, following the developments in \textnormal{\cite[Section 3.1]{arXiv:GondPougkPears}}. The aforementioned extensions are omitted here for brevity of exposition, but should be considered in a production implementation of the proposed method.
\end{remark}

\section{Applications and numerical results} \label{sec: applications}
\par In this section we present various applications that can be modeled by problem \eqref{primal problem}. Specifically, we consider portfolio optimization, $L^1$-regularized partial differential equation (PDE) constrained optimization, quantile regression, and binary classification problems. We first discuss (and numerically demonstrate) the effectiveness of the approach for the solution of single-period mean-risk portfolio optimization problems, where risk is measured via the \emph{conditional value at risk} or the \emph{mean absolute semi-deviation}. Subsequently, we apply the proposed scheme to solve $L^1$-regularized Poisson optimal control problems. Then, we test the proposed approach on quantile regression problems, and finally on binary classification problems via linear support vector machines. 
\paragraph{Implementation details} Before proceeding to the numerical results, we mention certain implementation details of the proposed algorithm. Following the discussion in Section \ref{sec: solution of linear systems}, our implementation takes problems of the following form as input:
\begin{equation*}
\begin{split}
\underset{x \in \mathbb{R}^n}{\text{min}}&\quad \left\{c^\top x + \frac{1}{2}x^\top Q x + \sum_{i = 1}^l \left( \left(C x + d\right)_{i}\right)_+ + \|Dx\|_1 + \delta_{\mathcal{K}}(x)\right\}, \\ 
\textnormal{s.t.} &\quad \ A x = b,
\end{split}
\end{equation*}
\noindent where $D \in \mathbb{R}^{n \times n}$ is a diagonal weight matrix. As discussed in Section \ref{sec: solution of linear systems}, the separable nonsmooth term $\|Dx\|_1$ is treated separately as shown in the derivation of the SSN linear systems. The MATLAB implementation can be found on GitHub\footnote{\url{https://github.com/spougkakiotis/Krylov_SSN_PMM}}. The experiments are run on a PC with a 2.3GHz Intel core i7-12700H processor, 32GB of RAM, using the Windows 11 operating system.
\par We initialize the PMM penalty parameters as $\beta_0 = 50$, and $\rho_0 = 100$. These are increased in subsequent iterations, at a rate that depends on the residual reduction achieved in that particular PMM iteration. When solving the PMM sub-problems using Algorithm \ref{primal-dual SNM algorithm}, we use a predictor-corrector-like heuristic in which the first iteration is accepted without line-search and then line-search is activated for subsequent iterations. Algorithm \ref{primal-dual PMM algorithm} is allowed to run for at most 200 iterations, while in each outer iteration Algorithm \ref{primal-dual SNM algorithm} is run for at most 40 steps. We utilize the minimum residual (MINRES) method \cite{PaigeSaundersSIAMNumAnal} for the solution of the associated SSN linear systems, which is allowed to run for at most 150 iterations. At the beginning of the optimization process the proposed scheme does not utilize the preconditioner given in \eqref{eqn: preconditioner for explicit SSN linear system}, unless MINRES reaches 100 iterations (without a preconditioner). Unless stated otherwise, a primal-dual iterate is accepted as optimal if the conditions given in \eqref{termination criteria for PD-PMM} are satisfied for the tolerance specified by the user. All other implementation details follow directly from the developments in Sections \ref{sec: PAL penalties}--\ref{sec: solution of linear systems}. We refer the reader to the implementation on GitHub for additional details.
\par In the presented experiments, we compare the proposed approach against the robust regularized interior point solver given in \cite{arXiv:GondPougkPears,COAP:PougkGond} (known as IP-PMM; \emph{interior point-proximal method of multiplier}s), the exact and inexact implementations of which can be found on GitHub\footnote{\url{https://github.com/spougkakiotis/IP_PMM}}$^,$\footnote{\url{https://github.com/spougkakiotis/IP-PMM_QP_Solver}}. There are several reasons why the comparison against IP-PMM is meaningful. Firstly, both methods can be classified as second-order solvers. Secondly, both solvers scale well with the problem size, since the inexact variant of IP-PMM also utilizes iterative linear algebra for the solution of its associated linear systems. In that case, both methods rely on a similar preconditioning strategy (we note that the preconditioner analyzed in \eqref{eqn: preconditioner for explicit SSN linear system} is a variant inspired from the preconditioner proposed in \cite[Section 3.1]{arXiv:GondPougkPears}, which is used in the benchmark inexact IP-PMM code). Finally, both methods are implemented in MATLAB. To ensure a fair comparison, we run both the exact and the inexact variants of IP-PMM and report the statistics of the best performing variant. To showcase the efficiency of both second-order solvers (active-set and IP-PMM) against potential first-order alternatives, in the first application (i.e., on the portfolio optimization instances) we will also report the results obtained by applying the well-known ADMM-based OSQP solver given in \cite{osqp}, the implementation of which can also be found on GitHub\footnote{\url{https://github.com/osqp/osqp-matlab}}. 
\subsection{Portfolio optimization}
\par We first consider the mean-risk portfolio selection problem (originally proposed in \cite{JoFinance:Markowitz}), where we minimize some convex risk measure, while keeping the expected return of the portfolio above some desirable level. A variety of models for this problem have been extensively analyzed and solved in the literature (e.g., see \cite{JBanFin:AlexColYi,JRisk:KroUryPal,EJOR:Lwin_etal,JRisk:RockUry}). The departure from the variance as a measure of risk often allows for great flexibility in the decision making of investors, enabling them to follow potential regulations as well as to better control the risk associated with an ``optimal" portfolio. Optimality conditions and existence of solutions for several general deviation measures have been characterized in \cite{MathProg:RockUryZaba}. Additionally, the comparison of portfolios obtained by minimizing different risk measures has also been considered multiple times (e.g. see \cite{InvAnJ:GilMeik,ssnr:Rametal,FinResLet:RigBor,Prod:Silva_etal}). 
\par The method presented in this paper is general and not related to a specific model choice. Indeed, we would like to showcase the efficiency of our approach for obtaining accurate solutions to portfolio optimization problems with two distinct risk measures of practical interest. To that end, we focus on the solution of a standard portfolio selection model that has been used in the literature. All numerical results are obtained on real-world datasets. As a result the problems are of medium-scale. Nonetheless, even for such medium-scale problems, we will be able to demonstrate the efficiency of the proposed approach, when compared to a robust interior point method employed in the literature \cite{SIREV:DeSimone_etal} for similar problems. We also note that both second-order solvers (i.e., IPM and active-set) significantly outperform OSQP on these instances, however the latter was included in the comparison for completeness. Some large-scale instances will be tackled in the following subsections.
\par We consider the case where a decision vector $x \in \mathbb{R}^n$ represents a portfolio of $n$ financial instruments, such that 
\[x_i \in \left[a_{l_i},a_{u_i}\right],\quad  \textnormal{with}\ a_{l_i} = 0,\ a_{u_i} \leq 1,\ \textnormal{for all}\ i = 1,\ldots,n,\quad \textnormal{and}\ \sum_{i=1}^n x_i = 1. \]
\noindent This requirement indicates that no short positions are allowed, and no more than $a_{u_i}\%$ of the total wealth can be invested in instrument $i$. Let $\bm{\xi} \in \mathbb{R}^n$ denote a random vector, the $i$-th entry of which represents the random return of the $i$-th instrument. Then, the negative of the random return of a given portfolio $x$ is given by $g(x,\bm{\xi}) \triangleq -x^\top \bm{\xi}$. In this paper we assume that $\bm{\xi}$ follows some continuous distribution $p(\bm{\xi})$, as well as that there is a one-to-one correspondence between percentage return and monetary value (as in \cite[Section 3]{JRisk:RockUry}). Additionally, given some expected benchmark return $r$ (e.g., the \emph{market index}), we only consider portfolios that yield an expected return above a certain threshold, i.e. $\mathbb{E} [-g(x,\bm{\xi})] \geq r.$
\par Finally, given the previously stated constraints, we would like to minimize some convex risk measure $\varrho(\cdot)$ of interest. By putting everything together, the model reads as
\begin{equation} \label{model: vanilla portfolio selection}
\begin{split}
\min_{x} &\ \varrho\left(g(x,\bm{\xi})\right),\\
\textnormal{s.t.}&\ \sum_{i=1}^n x_i = 1,\\
&\ \mathbb{E}\left[-g(x,\bm{\xi})\right] \geq r,\\
&\ x_i \in \left[a_{l_i},a_{u_i}\right],\qquad i = 1,\ldots,n.
\end{split}
\end{equation}
\par There are several methods for solving such stochastic problems; let us mention two important variants. There is the \emph{parametric approach} (e.g., as in \cite{JBanFin:AlexColYi,JRisk:RockUry}), where one assumes that the returns follow some known distribution which is subsequently sampled to yield finite-dimensional optimization problems, and the \emph{sampling approach} (e.g., as in \cite{JRisk:KroUryPal}), where one obtains a finite number of samples (without assuming a specific distribution). Such samples are often obtained by historical observations, and this approach is also followed in this paper. It is well-known that historical data cannot fully predict the future (see \cite{JRisk:KroUryPal}), however it is a widely-used practice. The reader is referred to \cite{EFM:Jorion} for an extensive study on probabilistic models for portfolio selection problems, with practical considerations. 
\par Additional soft or hard constraints can be included when solving a portfolio selection problem. Such constraints can either be incorporated directly via the use of a model (e.g., see \cite[Section 2]{NowPublishers:Boyd_etal}) as hard constraints or by including appropriate penalty terms in the objective (soft constraints). It is important to note that the model given in \eqref{primal problem} is quite general and as a result has great expressive power, allowing one to incorporate various important risk measures (and their combinations), as well as modeling constraints of interest.
\paragraph{Real-world datasets:} In what follows, we solve two different instances of problem  \eqref{model: vanilla portfolio selection}. In particular, we consider two risk measures; the conditional value at risk (e.g., see \cite{JRisk:RockUry}), as well as the mean absolute semi-deviation (e.g., see \cite[Section 6.2.2.]{SIAM:Shapiro}, noting that this is in fact equivalent to the mean absolute deviation originally proposed in \cite{ManagSci:KonnoYama}). 
\par We showcase the effectiveness of the proposed approach on 6 real datasets taken from \cite{DataBrief:Bruni_etal}. Each dataset contains time series for weekly asset returns and market indexes for  different major stock markets, namely, DowJones, NASDAQ100, FTSE100, SP500, NASDAQComp, and FF49Industries. In the first 5 markets the authors in \cite{DataBrief:Bruni_etal} provide the market indexes, while for the last dataset the uniform allocation strategy is considered as a benchmark. Additional information on the datasets is collected in Table \ref{Table: portfolio optimization datasets}. We note that the authors in \cite{DataBrief:Bruni_etal} disregard stocks with less than 10 years of observations.
\begin{table}[!ht]
\centering
\caption{\centering Portfolio optimization datasets.\label{Table: portfolio optimization datasets}}
\scalebox{1}{
\begin{tabular}{llll}     
\specialrule{.2em}{.05em}{.05em} 
    \textbf{Name} & \textbf{\# of assets} & \textbf{\# of data points} & \textbf{Timeline} \\ \specialrule{.1em}{.3em}{.3em} 
DowJones [DJ] & 28 & 1363 & Feb. 1990--Apr. 2016\\ 
NASDAQ100 [NDQ]& 82 & 596 & Nov. 2004--Apr. 2016\\ 
FTSE100 [FT]& 83 & 717 & Jul. 2002--Apr. 2016\\ 
FF49Industries [FF] & 49 & 2325 & Jul. 1969--Jul. 2015\\ 
SP500 [SP]& 442 & 595 & Nov. 2004--Apr. 2016\\ 
NASDAQComp [NDC] & 1203 & 685 & Feb. 2003--Apr. 2016\\ 
\specialrule{.2em}{.05em}{.05em} 
\end{tabular}}
\end{table}

\subsubsection{Conditional value at risk}
\par First, we consider portfolio optimization problems that seek a solution minimizing the conditional value at risk; a measure which is known to be \emph{coherent} (see \cite{MathFinance:Artzner} for a definition of coherent risk measures). In particular using the notation introduced earlier, we consider the following optimization problem
\begin{equation} \label{CVaR original problem}
\underset{x \in \mathbb{R}^n}{\text{min}} \left\{  \textnormal{CVaR}_{\alpha}\left(g(x,\bm{\xi})\right) +  \delta_{\mathcal{K}}(x)\right\}, \qquad \textnormal{s.t.}\ A x = b,
\end{equation}
\noindent where $g(x,\bm{\xi})$ is the random cost function, $A \in \mathbb{R}^{m\times n}$ models the linear constraint matrix of problem \eqref{model: vanilla portfolio selection} (where an auxiliary variable has been introduced to transform the inequality constraint involving the discretized expectation into an equality), and $\mathcal{K} \triangleq [a_l,a_u]$. In the above $1-\alpha \in (0,1)$ is the confidence level. It is well-known (\cite{JRisk:RockUry}) that given a continuous random variable $X$, the conditional value at risk can be computed as an infimal convolution
\begin{equation*} 
\textnormal{CVaR}_{\alpha}\left(X\right)\triangleq \min_{t \in \mathbb{R}}\left\{t + \alpha^{-1}\mathbb{E}\left[ \left(X - t\right)_+ \right]\right\}.
\end{equation*}
\noindent Thus, we can write problem \eqref{CVaR original problem} in the following equivalent form:
\begin{equation*}
\underset{(x,t) \in \mathbb{R}^n\times \mathbb{R}}{\text{min}} \left\{ t + \frac{1}{l \alpha} \sum_{i = 1}^l \left(-\xi_i^\top x - t\right)_+ + \delta_{\mathcal{K}}(x)\right\},\qquad \textnormal{s.t.}\ Ax = b,
\end{equation*}
\noindent where all expectations have been substituted by summations since we assume the availability of a dataset $\{\xi_1,\ldots,\xi_l\}$.
\par We solve problem \eqref{CVaR original problem} using the proposed active-set method (AS), the {interior point-proximal method of multipliers} (IP-PMM) given in \cite{COAP:PougkGond}, and OSQP (\cite{osqp}). We set $\texttt{tol} = 10^{-5}$, and run the three methods for each of the datasets described in Table \ref{Table: portfolio optimization datasets} for varying confidence level. We report the confidence parameter $\alpha$, the number of PMM, SSN, and Krylov iterations required by the active-set scheme, the number IP-PMM and Krylov iterations required by the interior point scheme (with Krylov iterations reported only if the inexact IP-PMM variant performs better), and the number of OSQP iterations, the CPU time needed by each of the three schemes, as well as the number of factorizations used within the active set (PMM-SSN) scheme. As already mentioned earlier, the AS scheme does not utilize a preconditioner until it is deemed necessary. Additionally, it often happens that the active-set is not altered from one iteration to the next, and the factorization of the preconditioner does not need to be re-computed. The results are collected in Table \ref{Table CVaR portfolio selection: varying confidence level}. In all the numerical results that follow, the lowest running time exhibited by a solver, assuming it successfully converged, is presented in bold. 
\begin{table}[!ht]
\centering
\scalebox{0.9}{\begin{threeparttable}\caption{\centering CVaR portfolio selection: varying confidence level (\texttt{tol} = $10^{-5}$).\label{Table CVaR portfolio selection: varying confidence level}}
\begin{tabular}{llllllll}     
\specialrule{.2em}{.05em}{.05em} 
    \multirow{2}{*}{\textbf{Dataset}} & \multirow{2}{*}{$\bm{\alpha}$} & \multicolumn{3}{c}{\textbf{Iterations}}     & \multicolumn{3}{c}{\textbf{Time (s)}}  \\   \cmidrule(l{2pt}r{2pt}){3-5} \cmidrule(l{2pt}r{2pt}){6-8}
&   &  {{PMM}(SSN)[Fact.]\{Krylov\}} & {{IP--PMM(Krylov)}} & {{OSQP}} &  {{AS}} & {{IP--PMM}}  & {{OSQP}}   \\ \specialrule{.1em}{.3em}{.3em} 
 \multirow{3}{*}{{DJ}} & $0.05$ & 29(78)[15]\{3841\} &  22(--)\tnote{1} &   32,575& \textbf{0.23}  & 0.91 & 10.94   \\
    & $0.10$ & 33(88)[24]\{4116\} &22(--) & 42,625 & \textbf{0.40} & 0.84  & 13.61 \\
        & $0.15$ & 36(103)[35]\{3954\} &19(--) & 43,100 & \textbf{0.54}  & 0.75 & 14.45 \\ \specialrule{.00002em}{.1em}{.1em}
 \multirow{3}{*}{NDQ} & $0.05$ & 26(80)[20]\{3722\} & 17(--) & 30,575 & \textbf{0.28}& 0.51 & 7.92   \\
     & $0.10$ & 29(104)[27]\{4781\} & 18(--) & 34,700 & \textbf{0.46} &  0.55 & 8.81 \\
    &  $0.15$ & 32(100)[33]\{4352\} & 17(--) & 40,500 & \textbf{0.45} &  0.53 & 10.46\\ \specialrule{.00002em}{.1em}{.1em}
\multirow{3}{*}{FE} & $0.05$ & 30(108)[73]\{2208\} & 21(--) &   21,825& \textbf{0.82}  & 1.02  & 10.96   \\
    & $0.10$ & 31(101)[46]\{3266\} & 16(--)  & 35,825 & \textbf{0.52} & 0.72 & 17.70 \\
        & $0.15$ & 34(99)[54]\{2798\} & 18(--) & 36,550 & \textbf{0.65}  & 0.93 & 18.24 \\ \specialrule{.00002em}{.1em}{.1em}
\multirow{3}{*}{FF} & $0.05$  & 38(118)[31]\{3491\} & 27(--) &  $50,000^{\ddagger}$\tnote{2} & \textbf{0.60}  & 3.80 & $59.16^{\ddagger}$  \\
     & $0.10$ & 37(132)[20]\{2777\}& 23(--) &  $50,000^{\ddagger}$ & \textbf{0.76}  & 3.21 & $57.92^{\ddagger}$\\
      & $0.15$ & 38(108)[19]\{4169\} & 20(--) & $50,000^{\ddagger}$ & \textbf{0.88}  & 2.82 &  $51.25^{\ddagger}$\\   \specialrule{.00002em}{.1em}{.1em}
 \multirow{3}{*}{SP} & $0.05$  & 29(117)[93]\{2753\} & 24(699)  & 21,075 & \textbf{1.87}  & 6.57 & 65.70  \\
    & $0.10$ & 32(119)[70]\{2861\} & 27(877) & 32,375 & \textbf{1.70}  & 7.79 & 99.39 \\
      & $0.15$ & 31(106)[64]\{3069\} & 23(--) &  39,225& \textbf{1.76}  & 13.49 & 119.95\\  \specialrule{.00002em}{.1em}{.1em}
     \multirow{3}{*}{NDC} & $0.05$ & 32(225)[154]\{4538\}& 26(593) & 16,100 & \textbf{6.89}  & 23.65 & 143.82  \\
     & $0.10$ & 31(168)[123]\{3441\} & 27(830) & 31,275 & \textbf{6.91}  &  27.54 & 291.97 \\
      & $0.15$ & 34(179)[124]\{3538\}& 28(925) & 16,100 & \textbf{8.65} & 31.43 & 403.75\\
\specialrule{.2em}{.05em}{.05em} 
\end{tabular}\begin{tablenotes}
\item[1] -- in the Krylov iterations for IP-PMM indicates that the exact variant of IP-PMM was used.
\item[2] $\ddagger$ indicates that the solver reached the maximum number of iterations.
\end{tablenotes}
\end{threeparttable}}
\end{table}
\par From Table \ref{Table CVaR portfolio selection: varying confidence level} we observe that for the smaller instances (DowJones, NASDAQ100, FTSE100, FF49Industries) both second-order methods perform well and are comparable in terms of CPU time. Nonetheless, even in this case, the proposed active-set solver requires less computations per iteration. Indeed, this can be seen by the fact that the two methods achieve similar times but the active-set scheme is performing significantly more SSN iterations. On the other hand, for the larger instances (SP500, NASDAQComp) the proposed active-set scheme outperforms the interior point method significantly. This is mostly due to the efficiency gained from the lower memory requirements and cheaper factorizations of the preconditioner associated with the active-set solver. We additionally observe that both second-order methods are robust with respect to the confidence level and consistently outperform OSQP, which struggles to find a 5-digit accurate solution. We note that OSQP could potentially be competitive for smaller tolerances (e.g., for finding a 3-digit accurate solution); however, the application under consideration dictates that an accurate solution is needed, since a small improvement in the portfolio output can translate into considerable profits in practice. Concerning IP-PMM, we note that the inexact variant is only elected for some of the largest instances (i.e., SP500 and NASDAQComp). This is expected, since these problems are small- to medium-scale. Nonetheless, large-scale instances will be tackled in the following subsections.
\subsubsection{Mean absolute semi-deviation}
\par Next, we consider portfolio optimization problems that seek a solution minimizing the mean absolute semi-deviation, which is also coherent. We consider the following optimization problem
\begin{equation} \label{MAD original problem}
\underset{x \in \mathbb{R}^n}{\text{min}} \left\{  \textnormal{MAsD}\left(g(x,\bm{\xi})\right) + \delta_{\mathcal{K}}(x)\right\}, \qquad \textnormal{s.t.}\ A x = b,
\end{equation}
\noindent where, given a continuous random variable $X$, the associated risk is defined as 
\begin{equation*}
\textnormal{MAsD}\left(X\right)\triangleq \mathbb{E}\left[ \left(X - \mathbb{E}[X]\right)_+\right] \equiv \frac{1}{2} \mathbb{E}\left|X- \mathbb{E}[X]\right|,
\end{equation*}
\noindent where the equivalence follows from \cite[Prospotion 6.1]{SIAM:Shapiro}. Given a dataset $\{\xi_1,\ldots,\xi_l\}$, problem \eqref{MAD original problem} can be written as
\begin{equation*} 
\begin{split}
\underset{x \in \mathbb{R}^n}{\text{min}} & \left\{ \frac{1}{l}\sum_{i = 1}^l \left(-\xi_i^\top x + \bm{\mu}^\top x\right)_+ + \|Dx\|_1 + \delta_{\mathcal{K}}(x)\right\},\\
\textnormal{s.t.}\quad\ \ &\  Ax = b,
\end{split}
\end{equation*}
\noindent where $\bm{\mu} \triangleq \frac{1}{l}\sum_{i= 1}^l \xi_i^\top x$. Note that this model is in the form of \eqref{primal problem}. We fix $\texttt{tol} = 10^{-5}$ and run the three methods on the 6 datasets. The results are collected in Table \ref{Table MAsD portfolio selection: varying confidence level}.
\begin{table}[!ht]
\centering
\scalebox{0.9}{\begin{threeparttable}\caption{\centering MAsD portfolio selection (\texttt{tol} = $10^{-5}$).\label{Table MAsD portfolio selection: varying confidence level}}
\begin{tabular}{lllllll}     
\specialrule{.2em}{.05em}{.05em} 
    {\textbf{Dataset}}  & \multicolumn{3}{c}{\textbf{Iterations}}     & \multicolumn{3}{c}{\textbf{Time (s)}}  \\   \cmidrule(l{2pt}r{2pt}){2-4} \cmidrule(l{2pt}r{2pt}){5-7}
&     {{PMM}(SSN)[Fact.]\{Krylov\}} & {{IP--PMM(Krylov)}} & {{OSQP}}  &  {{AS}} & {{IP--PMM}}  & {{OSQP}}   \\ \specialrule{.1em}{.3em}{.3em} 
{{DJ}}  & 48(105)[31]\{2686\} &  14(--) & 14,050 & {0.79} & \textbf{0.54}   & 4.60\\ \specialrule{.00002em}{.1em}{.1em}
 {NDQ}    & 41(119)[37]\{3436\} & 13(--) & 40,550 &  {0.44}& \textbf{0.28} & 10.72\\ \specialrule{.00002em}{.1em}{.1em}
{FE}    & 42(133)[53]\{3020\} & 12(--) & 47,125 & {0.69}  & \textbf{0.30} & 23.29\\ \specialrule{.00002em}{.1em}{.1em}
{FF}   & 52(100)[0]\{4289\} & 18(--) & 12,050& \textbf{0.77}  & 2.70 & 14.24\\   \specialrule{.00002em}{.1em}{.1em}
 {SP}    & 42(180)[65]\{4313\} &  19(--) & 41,275 & \textbf{2.06}  & 4.09 & 128.09\\  \specialrule{.00002em}{.1em}{.1em}
    {NDC}  & 46(176)[75]\{4569\} &  14(--) & 41,000 &  \textbf{6.74}  & 10.57 & 367.32 \\
\specialrule{.2em}{.05em}{.05em} 
\end{tabular}
\end{threeparttable}}
\end{table}
\par From Table \ref{Table MAsD portfolio selection: varying confidence level} we observe that both second-order schemes are, once again, robust and comparable for all instances. In this case, the larger instances (SP500, NASDAQComp) were solved in comparable time by both solvers. Nevertheless, it is important to note that the active-set scheme is the better choice, since it has significantly less memory requirements (this is due to solving a smaller problem formulation while also only utilizing information from the active sets when constructing its associated preconditioners). Notably the proposed method is robust (i.e., converges reliably to an accurate solution). Finally, as in the case of the CVaR instances, OSQP struggled to find accurate solutions at a reasonable number of iterations, making it a less efficient choice for such problems. In subsequent experiments, we will drop OSQP from our benchmarks since it performed significantly worse compared to both second-order solvers.
\subsection{PDE-constrained optimization} 
\par Let us now test the proposed methodology on some optimization problems with partial differential equation constraints. We consider optimal control problems of the following form:
\begin{equation} \label{generic inverse problem}
\begin{split}
\min_{\mathrm{y},\mathrm{u}} \     &\ \frac{1}{2}\| \rm{y} - \bar{\rm{y}}\|_{L^2(\Omega)}^2 + \frac{\alpha_1}{2}\|\rm{u}\|_{L^1(\Omega)}^2 + \frac{\alpha_2}{2}\|\rm{u}\|_{L^2(\Omega)}^2, \\
\text{s.t.}\ &\ \mathrm{D} \mathrm{y}(\bm{x}) + \mathrm{u}(\bm{x}) = \mathrm{g}(\bm{x}),\quad \mathrm{u_{a}}(\bm{x}) \leq \mathrm{u}(\bm{x})  \leq \mathrm{u_{b}}(\bm{x}),
\end{split}
\end{equation}
\noindent where $(\rm{y},\rm{u}) \in \text{H}^1(\Omega) \times \text{L}^2(\Omega)$, $\mathrm{D}$ is some linear differential operator, $\bm{x}$ is a $2$-dimensional spatial variable, and $\alpha_1,\ \alpha_2 \geq 0$ are the regularization parameters of the control variable. The problem is considered on a given compact spatial domain $\Omega$, where $\Omega \subset \mathbb{R}^{2}$ has boundary $\partial \Omega$, and is equipped with Dirichlet boundary conditions. The algebraic inequality constraints are assumed to hold a.e. on $\Omega$, while ${\rm u_a}$ and ${\rm u_b}$ may take the form of constants or functions of the spatial variables. 
\par We solve problem \eqref{generic inverse problem} via a discretize-then-optimize strategy. We employ the Q1 finite element discretization implemented in IFISS\footnote{\url{https://personalpages.manchester.ac.uk/staff/david.silvester/ifiss/default.htm}} (see \cite{IFISSACM,IFISSSIAMREVIEW}) which yields a sequence of $\ell_1$-regularized convex quadratic programming problems in the form of \eqref{primal problem}. We note that the discretization of the smooth parts of problem \eqref{generic inverse problem} follows a standarad Galerkin approach (e.g., see \cite{AMS:Trolzsch}), while the $L^1$ term is discretized by the \emph{nodal quadrature rule} as in \cite{COAP:SongChenYu} (an approximation that achieves a first-order convergence--see \cite{ESAIM:GerdDaniel}).
\par In what follows, we consider two-dimensional $L^1/L^2$-regularized Poisson optimal control problems. The problem is posed on $\Omega = (0,1)^2$. Following \cite[Section 5.1]{NLAA:PearsonPorcStoll}, we set constant control bounds $\mathrm{u}_a = -2$, $\mathrm{u}_b = 1.5$, and the desired state as $\bar{\mathrm{y}} = \sin(\pi x_1)\sin(\pi x_2)$. In order to ensure that both solvers exhibit a consistent behaviour over PDE-constrained optimization problems of arbitrary size, we enforce stopping conditions with \emph{absolute} tolerance $\texttt{tol} = 10^{-4}$. This is important in this case, since IP-PMM solves a smooth reformulation of the problem, and its relative tolerance criteria differ significantly from those of the active-set method when the discretized PDE size increases. In Table \ref{Table Poisson optimal control: varying grid-size and L1 regularization}, we fix the $L^2$ regularization parameter to the value $\alpha_2 = 10^{-2}$, and present the runs of the two methods (i.e. active-set, IP-PMM) for varying $L^1$ regularization (i.e. $\alpha_1$) as well as grid size. We report the size of the resulting discretized problems (before any reformulation), the value of $\alpha_1$, the number iterations required by each solver, as well as the total time to convergence. 
\begin{table}[!ht]
\caption{Poisson control: varying grid size and $L^1$ regularization ($\alpha_2 = 10^{-2}$).\label{Table Poisson optimal control: varying grid-size and L1 regularization}}
\centering
\scalebox{1}{\begin{threeparttable}
\begin{tabular}{llllll}     
\specialrule{.2em}{.05em}{.05em} 
    \multirow{2}{*}{\textbf{$\bm{n}$}} & \multirow{2}{*}{$\bm{\alpha_1}$} & \multicolumn{2}{c}{\textbf{Iterations}}     &  \multicolumn{2}{c}{\textbf{Time (s)}}  \\   \cmidrule(l{2pt}r{2pt}){3-4} \cmidrule(l{2pt}r{2pt}){5-6}
&   &  {{PMM(SSN)[Fact.]\{Krylov\}}} & {{IP-PMM(Krylov)}}   & AS & IP-PMM     \\ \specialrule{.1em}{.3em}{.3em} 
 \multirow{3}{*}{$8.45\cdot 10^3$} & $10^{-2}$ &  14(39)[28]\{1691\} & 13(1096)   &  \textbf{3.14} & 4.36  \\
  &   $10^{-4}$ &14(39)[23]\{1687\} &14(1253) &     \textbf{2.67} & 4.50 \\
  &   $10^{-6}$ &  14(39)[23]\{1686\} & 14(1256) &  \textbf{2.49} & 4.11 \\  \specialrule{.00002em}{.1em}{.1em}
 \multirow{3}{*}{$3.32\cdot 10^4$} & $10^{-2}$ & 16(34)[28]\{1295\} & 15(1147) &   \textbf{9.98}  & 18.97  \\
  &   $10^{-4}$ & 16(34)[24]\{1291\} & 15(1379)  &  \textbf{9.98} & 22.55  \\
  &   $10^{-6}$ & 16(34)[24]\{1291\} & 15(1381)  &  \textbf{9.82} & 23.40  \\ \specialrule{.00002em}{.1em}{.1em}
\multirow{4}{*}{$1.32\cdot 10^5$} & $10^{-2}$ & 18(31)[29]\{1149\} & 14(876)   &  \textbf{43.98} & 67.10 \\
  &   $10^{-4}$ & 18(31)[26]\{1143\} & 15(1111) &  \textbf{42.56} & 83.88\\
  &   $10^{-6}$ & 18(31)[26]\{1143\} & 15(1116)  &  \textbf{42.69} & 82.23  \\ \specialrule{.00002em}{.1em}{.1em}
    \multirow{3}{*}{$5.26\cdot 10^5$} & $10^{-2}$ & 18(32)[30]\{1251\} & 15(1166)  & \textbf{246.08} & 450.25  \\
  &   $10^{-4}$ & 18(32)[30]\{1250\} & 15(1171)   &  \textbf{218.24} & 743.47 \\
  &   $10^{-6}$ &18(32)[30]\{1250\} & 15(1174) &  \textbf{259.12}  & 728.32\\  \specialrule{.00002em}{.1em}{.1em}
    \multirow{3}{*}{$2.11\cdot 10^6$} & $10^{-2}$ & 21(38)[35]\{1535\} & 15(909)  &   \textbf{1199.34} & 2676.53 \\
  &   $10^{-4}$ & 21(38)[36]\{1534\}  & 16(1068) &   \textbf{1342.57} & 1878.91 \\
  &   $10^{-6}$ &  21(38)[36]\{1534\} & 16(1068) &   \textbf{1452.42} & 1612.56 \\
\specialrule{.2em}{.05em}{.05em} 
\end{tabular}
\begin{tablenotes}
\item[1] $\ddagger$ indicates that the solver reached the maximum number of iterations.
\item[2] $\dagger$ indicates that the solver ran out of memory.
\end{tablenotes}
\end{threeparttable}}
\end{table}
\par From Table \ref{Table Poisson optimal control: varying grid-size and L1 regularization} we observe that both the active-set and IP-PMM are competitive, with the active-set method achieving faster convergence in terms of CPU time. Let us note that the number of Krylov iterations required to obtain an absolute 4-digit solution by both solvers is comparable. Thus, the active-set approach is more efficient, since it requires less memory (as it solves smaller and better conditioned linear systems, and thus the inversion of the associated preconditioners is less resource-consuming). On the other hand, we can observe that both AS and IP-PMM have a very consistent behaviour for a wide range of values of the associated problem parameters. 
\par Next, we fix $\alpha_1 = 10^{-4}$ and the problem size, and run three solvers on PDE-constrained optimization instances with varying $L^2$ regularization. The results are collected in Table \ref{Table Poisson optimal control: varying tolerance and L2 regularization}, and showcase the robustness of the proposed active-set solver with respect to the problem parameters. Indeed, one can observe that the active-set scheme exhibits more consistent convergence behaviour compared to IP-PMM as we alter the problem parameters.
\begin{table}[!ht]

\centering
\scalebox{1}{\begin{threeparttable}\caption{Poisson control: varying $L^2$ regularization $(n,\alpha_1) = (1.32\cdot 10^5,10^{-4})$.\label{Table Poisson optimal control: varying tolerance and L2 regularization}}
\begin{tabular}{lllll}     
\specialrule{.2em}{.05em}{.05em} 
     \multirow{2}{*}{$\bm{\alpha_2}$} & \multicolumn{2}{c}{\textbf{Iterations}}   & \multicolumn{2}{c}{\textbf{Time (s)}}  \\   \cmidrule(l{2pt}r{2pt}){2-3} \cmidrule(l{2pt}r{2pt}){4-5}
  &  {{PMM(SSN)[Fact.]\{Krylov\}}} & {{IP-PMM(Krylov)}}   & AS & IP-PMM      \\ \specialrule{.1em}{.3em}{.3em}
       $10^{-1}$ & 18(31)[26]\{1155\} & 18(3979) &     \textbf{44.06} & 423.02  \\
     $10^{-4}$ & 18(31)[26]\{1142\} & 15(661) &      \textbf{77.29} & 106.84  \\
     $10^{-6}$ & 18(31)[26]\{1141\} & 15(607)    & \textbf{60.52} & 103.41  \\ 
      $0$ & 18(31)[26]\{1144\} & 15(499)   &\textbf{70.45} & 73.72 \\   
\specialrule{.2em}{.05em}{.05em} 
\end{tabular}\begin{tablenotes}
\item[1] $\ddagger$ indicates that the solver reached the maximum number of iterations.
\end{tablenotes}
\end{threeparttable}}
\end{table}

\subsection{Penalized quantile regression}
\par Let us now consider linear regression models of the following form
\[y_i = \beta_0 + \xi_i^\top \beta + \epsilon_i, \qquad i \in \{1,\ldots,l\}\]
\noindent where $\xi_i$ is a $d$-dimensional vector of covariates, $(\beta_0,\beta)$ are the regression coefficients and $\epsilon_i$ is some random error. A very popular problem in statistics is the estimation of the optimal coefficients, in the sense of minimizing a model of the following form:
\begin{equation} \label{eqn: linear regression model}
    \min_{(\beta_0,\beta)\ \in\ \mathbb{R} \times \mathbb{R}^d} \left\{\frac{1}{l} \sum_{i=1}^l \ell\left(y_i - \beta_0 - \xi_i^\top \beta\right) + \lambda p(\beta) \right\},
\end{equation}
\noindent where $\ell(\cdot)$ is some loss function and $p(\cdot)$ is a penalty function with an associated regularization parameter $\lambda \geq 0$. Following \cite{JRSS:Zou}, we consider the elastic-net penalty, 
\[p(\beta) \equiv p_{\tau}(\beta) \triangleq \tau \|\beta\|_1 + \frac{1-\tau}{2}\|\beta\|_2^2,\qquad 0\leq \tau \leq 1.\]
\noindent For the loss function, we employ the quantile loss
\begin{equation} \label{eqn: quantile loss}
\ell(w) \equiv \rho_{\alpha}(w) \triangleq \left(1-\alpha\right) w_- + \alpha w_+ = \frac{1}{2}\left(\lvert w\rvert + (2\alpha-1)w\right), \qquad 0 < \alpha < 1,
\end{equation}
\noindent where $w \in \mathbb{R}$. Notice that the case $\alpha = \frac{1}{2}$ yields the absolute loss. Letting $x = \begin{bmatrix} \beta_0 & \beta^\top\end{bmatrix}^\top$, and using Remark \ref{remark: l1 norm reformulation}, we can re-write problem \eqref{eqn: linear regression model} in the form of \eqref{primal problem}, as
\begin{equation*}
\begin{split}
\min_{x\ \in\ \mathbb{R}^{1+d}}&\ \left\{(\alpha-1)\mathds{1}_l^\top \left(Cx+d\right) + \frac{1}{2}x^\top Q x + \sum_{i=1}^l ((Cx+d)_i)_+  + \|Dx\|_1\right\},
\end{split}
\end{equation*}
\noindent where
\[C_{:,i} = -\frac{1}{l}[1\ \ \xi_i^\top],\ d_i = \frac{1}{l}y_i,\quad i = 1,\ldots,l,\qquad Q = \begin{bmatrix}
 0 & 0_{1,d}\\ 0_{d,1} & \lambda (1-\tau) I_d
\end{bmatrix},\qquad D = \begin{bmatrix}
 0 & 0_{1,d}\\ 0_{d,1} & \lambda \tau I_d
\end{bmatrix}.\]
\paragraph{Real-world datasets:} In what follows, we solve several instances of problem  \eqref{eqn: linear regression model}. We showcase the effectiveness of the proposed approach on 5 regression problems taken from the LIBSVM library (see \cite{CC01a}). Additional information on the datasets is collected in Table \ref{Table: quantile regression datasets}. 
\begin{table}[!ht] 
\centering
\caption{\centering Quantile regression datasets.\label{Table: quantile regression datasets}}
\scalebox{1}{
\begin{tabular}{lll}     
\specialrule{.2em}{.05em}{.05em} 
    \textbf{Name} & \textbf{\# of data points} & \textbf{\# of features} \\    \specialrule{.1em}{.3em}{.3em} 
space\_ga & 3107 & 6\\ 
abalone & 4177 & 8 \\ 
cpusmall & 8192 & 12 \\ 
cadata & 20,640 & 8 \\ 
E2006-tfidf & 16,087 &  150,360 \\ 
\specialrule{.2em}{.05em}{.05em} 
\end{tabular}}
\end{table}
\par We fix $\texttt{tol} = 10^{-4}$, $\lambda = 10^{-2}$, and $\tau = 0.5$, and run the two methods (active-set (AS), IP-PMM) on the 5 instances for varying quantile level $\alpha$. The results are collected in Table \ref{Table quantile regression: varying confidence level}.
\begin{table}[!ht]
\centering
\begin{threeparttable}\caption{\centering Quantile regression: varying confidence level (\texttt{tol} = $10^{-4}$, $\lambda = 10^{-2}$, $\tau = 0.5$).\label{Table quantile regression: varying confidence level}}
\begin{tabular}{llllll}     
\specialrule{.2em}{.05em}{.05em} 
    \multirow{2}{*}{\textbf{Problem}} & \multirow{2}{*}{$\bm{\alpha}$} & \multicolumn{2}{c}{\textbf{Iterations}}     & \multicolumn{2}{c}{\textbf{Time (s)}}  \\   \cmidrule(l{2pt}r{2pt}){3-4} \cmidrule(l{2pt}r{2pt}){5-6}
&   &  {{PMM}(SSN)[Fact.]\{Krylov\}} & {{IP--PMM}(Krylov)} &  {{AS}} & {{IP--PMM}}      \\ \specialrule{.1em}{.3em}{.3em} 
      \multirow{4}{*}{{space\_ga}} & $0.50$ &  32(36)[0]\{623\} &  18(--)  & \textbf{0.15} &  0.31  \\ 
  & $0.65$ &  32(37)[0]\{636\} &   20(--) & \textbf{0.13} & 0.38 \\ 
 & $0.80$ &  30(34)[0]\{583\} &  16(--) & \textbf{0.12} &  0.25  \\
   & $0.90$ &  29(31)[0]\{536\} &   22(--) & \textbf{0.11} &  0.35\\ \specialrule{.00002em}{.1em}{.1em}
 \multirow{4}{*}{{abalone}} & $0.50$ &  27(35)[0]\{603\} & 34(--) & \textbf{0.17} & 1.39 \\ 
  & $0.65$ &  27(31)[0]\{534\} &  31(--) & \textbf{0.15} & 1.13  \\ 
 & $0.80$ &  25(30)[0]\{492\} &   22(--)  & \textbf{0.12} & 0.85  \\
   & $0.90$ &  24(24)[0]\{367\} & 20(--) & \textbf{0.09} &   0.74  \\ \specialrule{.00002em}{.1em}{.1em}
    \multirow{4}{*}{{cpusmall}} & $0.50$ &  28(32)[0]\{538\} & 25(--)  & \textbf{0.26} &   2.02 \\ 
  & $0.65$ &  28(33)[0]\{577\} & 25(--) &  \textbf{0.29} &  2.01  \\ 
 & $0.80$ &  29(30)[0]\{518\} &  22(--) &  \textbf{0.24} & 1.81  \\
   & $0.90$ &  29(31)[0]\{487\} &  25(--) & \textbf{0.23} &   2.13 \\ \specialrule{.00002em}{.1em}{.1em}
 \multirow{4}{*}{{cadata}} & $0.50$ &  4(46)[0]\{860\} & 49(--)  & \textbf{0.27} &   11.40 \\ 
  & $0.65$ &  5(49)[0]\{933\} &   42(--)   & \textbf{0.31} &  10.51    \\ 
 & $0.80$ &  5(50)[0]\{861\} &   45(--)  & \textbf{0.30} &   10.60   \\
   & $0.90$ &  29(1405)[214]\{5722\} &  $\doublebarwedge$\tnote{1} & \textbf{8.95} &   $\doublebarwedge$  \\ \specialrule{.00002em}{.1em}{.1em}
\multirow{4}{*}{{E2006-tfidf}} & $0.50$ &  37(37)[0]\{363\} &   $\dagger$\tnote{2}  & \textbf{60.91} &  $\dagger$  \\ 
  & $0.65$ &  36(39)[0]\{405\}   & $\dagger$ & \textbf{111.66} & $\dagger$  \\ 
 & $0.80$ &  36(39)[0]\{399\}  & $\dagger$ & \textbf{118.28} &  $\dagger$   \\
   & $0.95$ &  35(37)[0]\{388\}  & $\dagger$ & \textbf{103.19} &  $\dagger$  \\ 
\specialrule{.2em}{.05em}{.05em} 
\end{tabular}
\begin{tablenotes}
\item[1] $\doublebarwedge$ indicates that the solver incorrectly identified the problem as infeasible.
\item[2] $\dagger$ indicates that the solver ran out of memory.
\end{tablenotes}
\end{threeparttable}
\end{table}
\par From Table \ref{Table quantile regression: varying confidence level} we observe that the two approaches are comparable for the smaller instances (space\_ga, abalone, and cpusmall). The active-set scheme significantly outperforms IP-PMM on the cadata problem, mainly due to its superior numerical stability and better conditioning of its linear systems. Indeed, this problem is highly ill-conditioned; this is reflected in the increased number of IP-PMM iterations needed to obtain a 4-digit accurate solution. The better conditioning of the linear systems associated with the active-set (compared to those appearing within IP-PMM) can also be observed from the fact that it did not utilize the preconditioner in most instances, enabling it to successfully converge matrix-free. Finally, we observe that for the large-scale instance (E2006-tfidf), both the exact and the inexact variants of IP-PMM crashed due to memory requirements.
\par Next, we fix $\texttt{tol} = 10^{-4},$ and $\alpha = 0.8$, and run the two methods for varying regularization parameters $\tau$ and $\lambda$. The results are collected in Table \ref{Table quantile regression: varying regularization}. We observe that both methods are robust for a wide range of parameter values. As before, the active-set scheme consistently solved the large-scale instance (E2006-tfidf) for a wide range of parameter values, without running into any memory issues.

\begin{table}[!ht]
\centering
\scalebox{0.94}{\begin{threeparttable}\caption{\centering Quantile regression: varying regularization (\texttt{tol} = $10^{-4}$, $\alpha = 0.8$).\label{Table quantile regression: varying regularization}}
\begin{tabular}{llllll}     
\specialrule{.2em}{.05em}{.05em} 
    \multirow{2}{*}{\textbf{Problem}} & \multirow{2}{*}{$(\bm{\tau},\bm{\lambda})$} & \multicolumn{2}{c}{\textbf{Iterations}}     & \multicolumn{2}{c}{\textbf{Time (s)}}  \\   \cmidrule(l{2pt}r{2pt}){3-4} \cmidrule(l{2pt}r{2pt}){5-6}
&   &  {{PMM}(SSN)[Fact.]\{Krylov\}} & {{IP--PMM}(Krylov)} & {{AS}} & {{IP--PMM}}     \\ \specialrule{.1em}{.3em}{.3em} 
      \multirow{4}{*}{{space\_ga}}& $(0.2,5\cdot 10^{-2})$ &  30(36)[0]\{585\} & 26(--)   & \textbf{0.13} &  0.40 \\
      &$(0.4,1\cdot 10^{-2})$ &  30(33)[0]\{568\} & 20(--)  & \textbf{0.12} &  0.32 \\
       &$(0.6,5\cdot 10^{-3})$ &  30(33)[0]\{654\} &  18(--)  & \textbf{0.14} & 0.29  \\
        &$(0.8,1\cdot 10^{-3})$ &  30(32)[0]\{695\} & 22(--)  & \textbf{0.12} &   0.37\\ \specialrule{.00002em}{.1em}{.1em}
 \multirow{4}{*}{{abalone}} & $(0.2,5\cdot 10^{-2})$ &  26(31)[0]\{422\} &  20(--)   & \textbf{0.12} &  0.76  \\
      &$(0.4,1\cdot 10^{-2})$ &  25(30)[0]\{485\} &    27(--) & \textbf{0.14} & 0.90 \\
       &$(0.6,5\cdot 10^{-3})$ &  26(26)[0]\{474\} &   34(--) & \textbf{0.12} &   1.27 \\
        &$(0.8,1\cdot 10^{-3})$ &   26(26)[0]\{622\} &   25(--) & \textbf{0.13} &  0.74  \\ \specialrule{.00002em}{.1em}{.1em} 
 \multirow{4}{*}{{cpusmall}} & $(0.2,5\cdot 10^{-2})$ &  28(35)[0]\{440\} & 19(--)   & \textbf{0.26} & 1.58  \\
      &$(0.4,1\cdot 10^{-2})$ &  29(29)[0]\{487\} & 17(--)  & \textbf{0.24} & 1.37    \\
       &$(0.6,5\cdot 10^{-3})$ &  29(30)[0]\{652\} & 17(--)   & \textbf{0.28} &  1.47  \\
        &$(0.8,1\cdot 10^{-3})$ &  29(29)[0]\{871\} &  19(--) & \textbf{0.30} &   1.54 \\ \specialrule{.00002em}{.1em}{.1em} 
 \multirow{4}{*}{{cadata}} & $(0.2,5\cdot 10^{-2})$ &  34(940)[265]\{4436\} &  50(--) &   \textbf{5.93} &  11.34    \\
      &$(0.4,1\cdot 10^{-2})$ &  6(56)[0]\{928\} &  50(--)  & \textbf{0.34} & 11.55  \\
       &$(0.6,5\cdot 10^{-3})$ &  6(69)[0]\{1274\} & 45(--)   & \textbf{0.40} &   10.18 \\
        &$(0.8,1\cdot 10^{-3})$ &  12(187)[0]\{3621\} &  33(--)  & \textbf{1.05} &  7.34  \\ \specialrule{.00002em}{.1em}{.1em}
\multirow{4}{*}{{E2006-tfidf}} & $(0.2,5\cdot 10^{-2})$ &  34(43)[0]\{406\} & $\dagger$\tnote{1} &  \textbf{61.64} &  $\dagger$ \\ 
  & $(0.4,1\cdot 10^{-2})$ &  36(39)[0]\{395\} & $\dagger$ &   \textbf{57.64} & $\dagger$  \\ 
 & $(0.6,5\cdot 10^{-3})$ &  36(40)[0]\{430\} & $\dagger$ & \textbf{60.10} &  $\dagger$  \\
   & $(0.8,1\cdot 10^{-3})$ &  35(36)[0]\{447\} & $\dagger$  & \textbf{84.41} &  $\dagger$ \\
\specialrule{.2em}{.05em}{.05em} 
\end{tabular}
\begin{tablenotes}
\item[1] $\dagger$ indicates that the solver ran out of memory.
\end{tablenotes}
\end{threeparttable}}
\end{table}

\subsection{Binary classification via linear support vector machines}
\par Finally, we are interested in training a binary linear classifier using regularized soft-margin linear support vector machines (SVMs), \cite{Springer:Vapnik}. More specifically, given a training dataset  $\{(y_i,\xi_i)\}_{i=1}^l$, where $y_i \in \{-1,1\}$ are the \emph{labels} and $\xi_i \in \mathbb{R}^d$ are the \emph{feature vectors} (with $d$ the number of features), we would like to solve the following optimization problem
\begin{equation} \label{eqn: linear SVM}
\min_{(\beta_0,\beta) \in \mathbb{R}\times\mathbb{R}^d} \left\{\frac{1}{l}\sum_{i=1}^l \left( 1- y_i\left(\xi_i^\top \beta -\beta_0\right)\right)_+  + \lambda \left(\tau_1 \|\beta\|_1 + \frac{\tau_2}{2}\|\beta\|_2^2\right)\right\},
\end{equation}
\noindent where $\lambda > 0$ is a regularization parameter, and $\tau_1,\ \tau_2 > 0$ are the weights for the $\ell_1$ and $\ell_2$ regularizers, respectively. The standard Euclidean-norm regularization is traditionally used as a trade-off between the margin of the classifier (the larger the better) and the correct classification of $\xi_i$, for all $i \in \{1,\ldots,l\}$ (e.g., see \cite{MachLear:CortesVapnik}). However, this often leads to a dense estimate for $\beta$, which has led researchers in the machine learning community into considering the $\ell_1$ regularizer instead, in order to encourage sparsity (e.g., see \cite{ICML:BradManga}). It is well-known that both regularizers can be combined to deliver the benefits of each of them, using the elastic-net penalty (see, for example, \cite{StatSin:Wang_etal}), assuming that $\tau_1$ and $\tau_2$ are appropriately tuned.

\paragraph{Real-world datasets:} In what follows we consider elastic-net SVM instances of the form of \eqref{eqn: linear SVM}. We showcase the effectiveness of the proposed active-set scheme on 3 large-scale binary classification datasets taken from the LIBSVM library (\cite{CC01a}). The problem names, as well as the number of features and training points are collected in Table \ref{Table: Binary classification datasets}.
\begin{table}[!ht] 
\caption{\centering Binary classification datasets.\label{Table: Binary classification datasets}}
\centering
\scalebox{1}{
\begin{tabular}{lll}     
\specialrule{.2em}{.05em}{.05em} 
    \textbf{Name} & \textbf{\# of training points} & \textbf{\# of features} \\  \specialrule{.1em}{.3em}{.3em} 
rcv1 & 20,242 & 47,236 \\ 
real-sim & 72,309 & 20,958\\
news20 & 19,996 &  1,355,191\\
\specialrule{.2em}{.05em}{.05em} 
\end{tabular}}
\end{table}
\par In the experiments to follow, we only consider large-scale problems that neither the exact nor the inexact version of IP-PMM can solve on a personal computer, due to excessive memory requirements. Nonetheless, we should note that, by following the developments in \cite{SIREV:DeSimone_etal}, IP-PMM can be specialized to problems of this form in order to be able to tackle such large-scale instances. However, the same applies to the proposed active-set method, and thus specialization is avoided for the sake of generality. We fix $\texttt{tol} = 10^{-5}$, $\lambda = 10^{-2}$, and run the active-set solver for all datasets given in Table \ref{Table: Binary classification datasets} for varying values of the regularization parameters $\tau_1$, and $\tau_2$. The results are collected in Table \ref{Binary classification via elastic-SVMs: varying regularization }.
\begin{table}[!ht]
\caption{\centering Elastic-net SVMs: varying regularization (\texttt{tol} = $10^{-5}$, $\lambda = 10^{-2}$).\label{Binary classification via elastic-SVMs: varying regularization }}
\centering
\scalebox{1}{\begin{threeparttable}
\begin{tabular}{lllll}     
\specialrule{.2em}{.05em}{.05em} 
    \multirow{2}{*}{\textbf{Problem}} & \multirow{2}{*}{$\bm{\tau_1}$} & \multirow{2}{*}{$\bm{\tau_2}$} & \multicolumn{1}{c}{\textbf{Iterations}}     & \multicolumn{1}{c}{\textbf{Time (s)}}  \\   \cmidrule(l{2pt}r{2pt}){4-4} \cmidrule(l{2pt}r{2pt}){5-5}
&  &  &  {{PMM}(SSN)[Fact.]\{Krylov\}}    &  {{AS}}    \\ \specialrule{.1em}{.3em}{.3em} 
     \multirow{4}{*}{{rcv1}}&  0.2 & 0.2 &  46(64)[0]\{1497\} & 9.28\\
      & 0.8 & 0.2 &  42(50)[0]\{642\} & 6.90 \\
          &  0.2 & 0.8 &  44(68)[0]\{1228\} & 8.98 \\
      & 5 & 5 &  30(20)[0]\{138\} & 3.74 \\ \specialrule{.00002em}{.1em}{.1em}
              \multirow{4}{*}{{real-sim}}&  0.2 & 0.2 &  57(88)[0]\{2382\} & 34.99 \\
      &  0.8 & 0.2 &  38(22)[0]\{180\} & 10.42 \\
     &  0.2 & 0.8  &  55(84)[0]\{1740\} & 31.21\\
      & 5 & 5 &  37(26)[0]\{168\} & 10.66 \\  \specialrule{.00002em}{.1em}{.1em}
               \multirow{4}{*}{{news20}} &  0.2 & 0.2 &  36(43)[0]\{629\} & 70.22 \\
      &  0.8 & 0.2 &  29(26)[0]\{206\} & 40.29 \\
         &  0.2 & 0.8  &  31(110)[31]\{857\} & 142.18 \\
      & 5 & 5 &  33(26)[0]\{154\} & 39.94 \\
\specialrule{.2em}{.05em}{.05em} 
\end{tabular}
\end{threeparttable}}
\end{table}

\par We observe that the solver was consistently able to solve these instances, without running into memory issues. In this case, the behaviour of the active-set solver was affected by the parameters $\tau_1$ and $\tau_2$ (indeed, see the number of SSN iterations for different regularization values), however, we consistently obtain convergence in a very reasonable amount of time. We observe that no preconditioning was required for most of these problems, except from one instance originating from the news20 dataset (with $\tau_1 = 0.2$ and $\tau_2 = 0.8$). On the one hand, even when the preconditioner was required, the method did not run into any memory issues. On the other hand, we observe that the associated linear systems of the proposed method are relatively well-conditioned, allowing us to avoid the inversion of the preconditioner is many instances, which enables its matrix-free utilization in several settings. 
\par Overall, we observe that the proposed algorithm is very general and can be applied in a plethora of very important applications arising in practice. We were able to showcase that the active-set nature of the method allows one to solve large-scale instances on a personal computer. The proposed method strikes a good balance between first-order methods, which are fast but unreliable, and second-order interior point methods, which are extremely robust but can struggle with the problem size, ill-conditioning, and memory requirements. We have demonstrated that convex quadratic problems with piecewise-linear terms can be solved very efficiently using the proposed active-set scheme, and we conjecture that the algorithm can be readily extended to deal with significantly more general problems, either by appropriately extending its associated theory, or by adjusting its versatile numerical linear algebra kernel.

\section{Conclusions} \label{sec: Conclusions}
\par We presented an active-set method for the solution of convex quadratic optimization problems with piecewise-linear terms in the objective, with applications to sparse approximations and risk minimization. After showing the convergence of the proposed algorithm under minimal assumptions, we derived and analyzed a robust numerical linear algebra kernel for the solution of its associated linear systems. We derive preconditioners suitable for Krylov subspace methods that are robust with respect to the penalty parameters of our method and can, in general, pave the way for the widespread utilization of iterative linear algebra in the context of active-set schemes. We were able to numerically demonstrate the efficiency, scalability and robustness of the proposed approach over a wide range of instances arising from risk-averse portfolio selection, $L^1$-regularized partial differential equation constrained optimization, quantile regression and linear classification via support vector machines, showcasing its competitiveness against alternative state-of-the-art methods. We have demonstrated the generality and versatility of the proposed solver and its associated linear algebra kernel, which make it a very competitive alternative to first- or second-order schemes appearing in the literature. 
\appendix
\section{Appendix}
\subsection{Derivation of the dual problem} \label{Appendix: derivation of dual}
\par Assume that notation introduced in Section \ref{subsec: derivation of the PMM}. Let $y = [y_1^\top\ y_2^\top]^\top \in \mathbb{R}^{m+l}$. The dual of \eqref{primal problem} is $ \sup_{y,z}\inf_x \left\{ \ell(x,y,z)\right\}.$  We have
\begin{equation*}
\begin{split}
\inf_{x}\left\{\ell(x,y,z)\right\} =&\ -\sup_x \left\{\left(F^\top y-z\right)^\top x -f(x) \right\} + y^\top \hat{b} - \delta_{\mathcal{K}}^*(z)  -h^*(y_2)\\
=&\ -f^*(F^\top y - z) + y^\top \hat{b} - \delta_{\mathcal{K}}^*(z)-h^*(y_2)\\
=&\ -\frac{1}{2}x^\top Q x - \delta_{\{0\}}\left(c + Qx - F^\top y + z\right) + y^\top \hat{b} - \delta_{\mathcal{K}}^*(z) -h^*(y_2),
\end{split}
\end{equation*}
\noindent where we used the definition of the convex conjugate and the definition of $f(\cdot)$, which yields $f^*(F^\top y - z) = \frac{1}{2}x^\top Q x + \delta_{\{0\}}\left(c + Qx - F^\top y + z\right).$ In turn, we observe that this is exactly the dual problem \eqref{dual problem}.

\subsection{Derivation of the augmented Lagrangian} \label{Appendix: derivation of the augmented Lagrangian}
Let $\beta > 0$ be some penalty parameter. The augmented Lagrangian corresponding to \eqref{primal problem} reads:
\begin{equation*} 
\begin{split}
\mathcal{L}_{\beta}(x; y, z) &\triangleq\ \sup_{u',v'} \bigg\{\ell(x,u',v') - \frac{1}{2\beta}\|u'-y\|^2 - \frac{1}{2\beta}\|v'-z\|^2 \bigg\} \\
&=\ f(x) - \inf_{u'} \left\{u'^\top  \left(Fx -\hat{b}\right) + h^*(u_{m+1:m+l}) + \frac{1}{2\beta}\|u'-y\|^2 \right\} \\&\quad \ - \inf_{v'} \left\{-v'^\top x + \delta^*_{\mathcal{K}}(v') + \frac{1}{2\beta}\|v'-z\|^2 \right\} \\
&=\ f(x) -y_1^\top (Ax-b) + \frac{\beta}{2}\|Ax-b\|^2 \\&\quad \ - \inf_{w'}\left\{-w'^\top(Cx+d) + \delta_{\mathcal{C}}(w') + \frac{1}{2\beta}\|w'-y_2\|^2\right\} \\&\quad \ - \inf_{v'} \left\{-v'^\top x + \delta^*_{\mathcal{K}}(v') + \frac{1}{2\beta}\|v'-z\|^2 \right\}.
\end{split}
\end{equation*}
\noindent where $\mathcal{C} \triangleq \left\{w \in \mathbb{R}^l \big\vert w_i \in [0,1],\ j = 1,\ldots,l \right\}.$ Let us now solve each of the two minimization problems appearing in the last expression. For the first expression, we use the fact that if $p_1(x) = p_2(x) + r^\top x,$ where $p_1,\ p_2$ are two proper closed convex functions, then $\textbf{prox}_{p_1}(x) = \textbf{prox}_{p_2}(x-r).$ Thus, we obtain 
\[ \arg\min_{w'}\left\{-w'^\top(Cx+d) + \delta_{\mathcal{C}}(w') + \frac{1}{2\beta}\|w'-y_2\|^2\right\} = \Pi_{\mathcal{C}}\left(y_2 + \beta(Cx+d)\right).\]
\noindent For the second minimization problem, we proceed as before to obtain that 
\[ \arg\min_{w'}\left\{-v'^\top x + \delta^*_{\mathcal{K}}(v') + \frac{1}{2\beta}\|v'-z\|^2 \right\} = \textbf{prox}_{\beta \delta_{\mathcal{K}}^*}\left(z+\beta x\right).\]
\noindent Then, by using the \emph{Moreau Identity} (see \cite{BSMF:Moreau}):
\begin{equation} \label{Moreau Identity}
\textbf{prox}_{\beta p}(u') + \beta \textbf{prox}_{\beta^{-1}p^*}(\beta^{-1}u') = u',
\end{equation}
\noindent and by performing some manipulations, we arrive at the final expression for the augmented Lagrangian given in \eqref{final augmented lagrangian of the primal}.
\subsection{Differentiability of the augmented Lagrangian} \label{Appendix: differentiability of augmented Lagrangian}
\par In order to find the gradient of the augmented Lagrangian, we need to use the gradient expression of the Moreau envelope of a given proper closed convex function $p$. Specifically, given some $\beta > 0$, the Moreau envelope of $p$ reads
\[ e_{\beta p}(x) = \inf_{w} \left\{p(w) + \frac{1}{2\beta}\|w-x\|^2\right\},\]
\noindent and it can be shown that $\nabla e_{\beta p}(x) = \frac{1}{\beta}\left(x - \textbf{prox}_{\beta p}(x)\right).$ We proceed by computing the gradient of 
\[ \inf_{w'}\left\{-w'^\top(Cx+d) + \delta_{\mathcal{C}}(w') + \frac{1}{2\beta}\|w'-y_2\|^2\right\},\]
\noindent noting that the gradient of the second Moreau envelope (implicitly) appearing in the augmented Lagrangian can be computed similarly (using, additionally, the Moreau identity). We notice that
\begin{equation*}
\begin{split}
 &\inf_{w'}\left\{-w'^\top(Cx+d) + \delta_{\mathcal{C}}(w') + \frac{1}{2\beta}\|w'-y_2\|^2\right\} \\ &\quad =  \inf_{w'}\left\{\delta_{\mathcal{C}}(w') + \frac{1}{2\beta}\left\|w'-\left(y_2 + \beta(Cx+d)\right)\right\|^2\right\} -\frac{\beta}{2}\|Cx+d\|^2 - y_2^\top(Cx+d).
 \end{split}
 \end{equation*}
 \noindent Thus, using the expression of the gradient of the second (shifted) Moreau envelope alongside the chain rule, we obtain that
\begin{equation*}
\begin{split}&\nabla_x\left(\inf_{w'}\left\{\delta_{\mathcal{C}}(w') + \frac{1}{2\beta}\left\|w'-\left(y_2 + \beta(Cx+d)\right)\right\|^2\right\} -\frac{\beta}{2}\|Cx+d\|^2 - y_2^\top(Cx+d)\right)\\ &\quad = -C^\top \Pi_{\mathcal{C}}\left(y_2 + \beta(Cx+d)\right). 
\end{split}
 \end{equation*}
 \noindent Proceeding in a similar manner for the second Moreau envelope, we arrive at the expression for the gradient of the proximal augmented Lagrangian penalty function given in Section \ref{subsec: derivation of the PMM}.

\subsection{Termination criteria} \label{Appendix: termination criteria}
\par Let $\mathcal{C} = \left\{x \in \mathbb{R}^l \vert\ x_i \in [0,1],\ i \in \{1,\ldots,l\}\right\}$. We write the optimality conditions for \eqref{primal problem}--\eqref{dual problem} as
\begin{equation} \label{optimality conditions for P-D}
c + Qx - A^\top y_1 + C^\top y_2 + z = 0_{n},\qquad Ax = b,\qquad y_2 = \Pi_{\mathcal{C}}\left(y_2 
+ Cx + d\right),\qquad x = \Pi_{\mathcal{K}}(x + z),
\end{equation}
\noindent and the termination criteria for Algorithm \ref{primal-dual PMM algorithm} (given a tolerance $\epsilon > 0$) are set as
\begin{equation} \label{termination criteria for PD-PMM}
\max\left\{\frac{\|c + Qx - A^\top y_1 + C^\top y_2 + z\|}{1+\|c\|}, \frac{\left\|\begin{matrix} Ax - b \\ y_2 - \Pi_{\mathcal{C}}(y_2 + Cx+d) \end{matrix}\right\|}{1+\left\|\begin{matrix} b\\d\end{matrix}\right\|}, \left\|x - \Pi_{\mathcal{K}}(x + z)\right\| \right\} \leq \epsilon.
\end{equation}
\bibliography{references} 
\bibliographystyle{siamplain}

\end{document}